\documentclass[review]{elsarticle}

\usepackage{a4wide,amssymb,graphics,graphicx,textcomp}
\usepackage{enumerate,textcomp,multirow,amsmath}
\usepackage{algorithm}
\usepackage{algorithmic}
\usepackage{url}
\usepackage{amsmath, amsthm}
\usepackage[colorlinks]{hyperref}
\usepackage{mathrsfs}
\numberwithin{equation}{section}

\newtheorem{theorem}{\bf Theorem}[section]

\newtheorem{definition}[theorem]{\bf Definition}
\newtheorem{corollary}[theorem]{\bf Corollary}

\newtheorem{remark}[theorem]{\bf Remark}
\newtheorem{lemma}[theorem]{\bf Lemma}
\newtheorem{example}[theorem]{\bf Example}

\def\imag{\mathop{\mathrm{Im}}}

\newcommand {\mat}  [1] {\left[\begin{array}{#1}}
\newcommand {\rix}      {\end{array}\right]}

\newcommand{\eproof}{\space
    {\ \vbox{\hrule\hbox{\vrule height1.3ex\hskip0.8ex\vrule}\hrule}}\par}

\def\diag{\mathop{\mathrm{diag}}}

\newcommand{\nrm}[1]{{\left\vert\kern-0.25ex\left\vert\kern-0.25ex\left\vert #1 
    \right\vert\kern-0.25ex\right\vert\kern-0.25ex\right\vert}}

\newcommand{\C}{{\mathbb C}}
\newcommand{\R}{{\mathbb R}}

\begin{document}

\begin{frontmatter}


\title{Structured eigenvalue backward errors for rational matrix polynomials with symmetry structures}

\author{Anshul Prajapati\fnref{iitd}}
\author{Punit Sharma\fnref{iitd}}
\fntext[iitd]{Department of Mathematics, Indian Institute of Technology Delhi, Hauz Khas, 110016, India; \texttt{\{maz198078, punit.sharma\}@maths.iitd.ac.in.}
A.P. acknowledges the support of the CSIR Ph.D. grant by Ministry of Science \& Technology, Government of India. P.S. acknowledges the support of the DST-Inspire Faculty Award (MI01807-G) by Government of India,  India.
}

\begin{abstract}
We derive computable formulas for the structured backward errors of a complex number $\lambda$ when considered as an approximate eigenvalue of rational matrix polynomials that carry a symmetry structure. We consider symmetric, skew-symmetric, T-even, T-odd, Hermitian, skew-Hermitian, $*$-even, $*$-odd, and $*$-palindromic structures.  
Numerical experiments show that the backward errors with respect to structure-preserving and arbitrary perturbations are significantly different.
\end{abstract}


\end{frontmatter}


\section{Introduction}

In this paper, we study the perturbation analysis of the \emph{rational eigenvalue problem} (REP)
\begin{equation}\label{rep}
    G(z)x=0,
\end{equation}
where $G(z)$ is an $n\times n$ \emph{rational matrix polynomial} (RMP) of the form 
\begin{equation}\label{rmatrix}
    G(z)=\sum_{p=0}^{d}z^pA_p+\sum_{j=1}^k \frac{s_j(z)}{q_j(z)}E_j,
\end{equation}
where $A_p$'s and $E_j$'s are $n\times n$ constant matrices that carry some symmetry structure, and $s_j(z)$ and $q_j(z)$ are scalar polynomials with specific properties. If a pair $(\lambda,x) \in \C \times \C^{n}\setminus \{0\}$ satisfies~\eqref{rep}, then $\lambda$ is called an \emph{eigenvalue} of $G(z)$ and $x$ is called a corresponding \emph{eigenvector}. 
For a given RMP $G(z)$  with a symmetry structure and  $\lambda \in \C$, we are interested in finding the nearest RMP to $R(z)$ with the same symmetry structure of which $\lambda$ is an eigenvalue.

The REPs arise in a wide range of applications such as in acoustic emissions of high-speed trains, loaded elastic string, damped vibration of structures, electronic structure calculations of quantum dots and in control theory see \cite{MR3031626,MR2124762,MR2249422,MR2054343,MR2260626}. The REPs with $G(z)$ in the form~\eqref{rmatrix} are special cases of general nonlinear eigenvalue problems of the form
\[
\mathcal M(z)x=0 \quad \text{with} \quad \mathcal M(z)= \sum_{j=1}^m M_jf_j(z),
\]
where $M_j$'s are $n \times k$ matrices and $f_1(z),\ldots,f_m(z)$ are scalar-valued functions; see~\cite{MR2124762,TisM01} for a survey on a large number of applications and~\cite{BetHMST13} for some benchmark examples. The rational polynomials that occur in applications often satisfy a symmetry structure. 
For examples, 
\begin{example}{\rm \cite{BChoHHL11}} {\rm 
The numerical solution of a fluid-structure interaction leads to a symmetric REP of the form 
\[
\left( \frac{z^2}{a^2} M + K + \frac{z^2}{z \beta + \alpha} D 
\right)x=0,
\]
where $a$ is the speed of sound in the given material, $\alpha$ and $\beta$ are positive constants, $M$ and $K$ are symmetric positive definite matrices, respectively describing the mass and the stiffness, and the matrix $D$ is symmetric positive semidefinite describing the damping effects of an absorbing wall. 
}
\end{example} 

\begin{example}{\rm 
The REP arising in the finite element simulation of mechanical problems, see~\cite{Sol06,Vos03} , have the following form:
\[
\left ( P(z) +Q (z) \sum_{j=1}^m \frac{z}{z-\sigma_j}E_j
\right )x=0,
\]
where $P(z)$ and $Q(z)$ are symmetric matrix polynomials with large and sparse coefficients, and $E_j$'s are low-rank matrices. For a classical example, the REP
\[
\left(A-zB + \sum_{j=1}^k\frac{z}{z-\sigma_j}E_j\right)x=0,
\]
where $\sigma_j$ are positive, $A$ and $B$ are symmetric positive definite matrices of size $n\times n$ and $E_j=C_jC_j^T$ with $C_j \in \mathbb{R}^{n\times r_j}$ are of rank $r_j$, arises in the study of the simulation of mechanical vibration of fluid-solid structures~\cite{MR2249422,MR2124762,MehV06}.
}
\end{example} 

As mentioned in~\cite{MazV06,HwaLW05}, in some cases, the REPs can be transformed into a higher degree polynomial eigenvalue problem by clearing out the denominators, and the polynomial eigenvalue problem can then be linearized into a  linear eigenvalue problem. However, the size of the problem may substantially increase, and extra un-physical eigenvalues are typically introduced. These have to be recognized and removed from the computed spectrum. This makes it essential to study the perturbation analysis of the REPs in its original form without turning it into a polynomial problem. 

Linearization of a rational matrix polynomial $G(z)$ is a very recent concept and some classes of linearizations and strong linearizations are recently developed, see~\cite{MR3477318,MR3878309,MR2811297}. If the RMP carries a symmetry structure, then the linearizations that preserve the symmetry structure are prefered; see~\cite{MR3908736} for structured strong linearizations and~\cite{MR3910502} for strong minimal linearizations. 

Backward errors play an important role in the stability of an algorithm that  computes eigen-elements of a matrix function. Suppose matrix functions with additional symmetry structures are considered, then in that case, structure-preserving algorithms are advisable because, in this way, existing symmetries in the spectrum are preserved even under round-off errors. Although the eigenvalue/eigenpair backward errors have been well studied for matrix polynomials, see~\cite{MR2780396,MR2496422,MR3194659,MR3335496,MR4404572,Tis00}, there is only a little literature that deals with the perturbation analysis of rational or more general nonlinear eigenvalue problems, see~\cite{MR2124762,Ruh73,MR3568297,Ahm19}. In~\cite{BChoHHL11}, Ahmad and Mehrmann, have studied eigenpair backward errors of RMPs with various structures.  However, there is also a need for the backward error formulas for eigenvalues of structured rational matrix polynomials. Indeed if one is interested in computing the eigenvalues but not in the eigenvectors or invariant subspaces, then the corresponding error analysis should consider this. 

Our work is mainly motivated by~\cite{MR3568297} and aims at deriving computable formulas for the eigenvalue backward errors of  structured RMPs with respect to structure-preserving perturbations. For this, we generalize the framework developed in~\cite{MR3194659,MR3335496} for the eigenvalue backward error formulas for matrix polynomials with Hermitian and palindromic structures. The structures we consider on a RMP $G(z)$ are listed in Table~\ref{tab:my_label}. 

\begin{table}[]
    \centering
    \begin{tabular}{|c|c|c|}\hline 
        \textbf{Structure} &  \textbf{structure on $A_p$} & \textbf{structure on $E_j$}\\ 
    \textbf{on $G(z)$}    & \textbf{$p=0,\ldots,d$} & \textbf{$j=1,\ldots,k$} \\ \hline
        symmetric  &  $A_p^T=A_p$ & $E_j^T=E_j$  \\ \hline 
        skew-symmetric & $A_p^T=-A_p$ & $E_j^T=-E_j$ \\ \hline
        
T-even  & $A_p^T=(-1)^pA_p$ &  
\begin{tabular}{c c}
$E_j^T=-E_j$ ~ if $w_j(-z)=-w_j(z)$\\ $E_j^T=E_j$ ~ if $w_j(-z)=w_j(z)$   
\end{tabular}\\ \hline
                      
 T-odd & $A_p^T=(-1)^{p+1}A_p$ & 
 \begin{tabular}{c c}
$E_j^T=-E_j$ ~if $w_j(-z)=w_j(z)$\\ $E_j^T=E_j$ ~ if $w_j(-z)=-w_j(z)$  \end{tabular}\\ \hline

 Hermitian & $A_p^*=A_p$ & $E_j^*=E_j$ ~ if $w_j(z)^*=w_j(\overline{z})$ \\ \hline
skew-Hermitian & $A_p^*=-A_p$ & $E_j^*=-E_j$ ~ if $w_j(z)^*=w_j(\overline{z})$ \\ \hline
$*$-even & $A_p^*=(-1)^pA_p$ & \begin{tabular}{c c}
$E_j^*=E_j$ ~ if $(w_j(-z))^*=w_j(\overline{z})$\\ 
$E_j^*=-E_j$ ~ if $(w_j(-z))^*=-w_j(\overline{z})$  
\end{tabular} \\ \hline 

$*$-odd  & $A_p^*=(-1)^{p+1}A_i$ & \begin{tabular}{c c}
$E_j^*=E_j$ ~ if $(w_j(-z))^*=-w_j(\overline{z})$\\ 
$E_j^*=-E_j$ ~ if $(w_j(-z))^*=w_j(\overline{z})$  
\end{tabular} \\ \hline 

$*$-palindromic & $A_p^*=A_{d-p}$ & $E_j^*=E_j$ \quad if $(w_j(z))^*={\overline z}^d w_j(\frac{1}{\overline{z}})$\\ \hline
T-palindromic & $A_p^T=A_{d-p}$ & $E_j^T=E_j$ \quad if $w_j(z)=z^dw_j(\frac{1}{z})$\\ \hline
    \end{tabular}
    \caption{Structures on G(z)}
    \label{tab:my_label}
\end{table}

This paper is organized as follows: In Section~\ref{sec:prelims}, we introduce definitions and  review some preliminary results that will be used for the main results in the paper. In Section~\ref{sec:sym}, we consider the structured eigenvalue backward error of a complex number $\lambda$ for symmetric and skew-symmetric RMPs. These formulas are extended for T-even and T-odd structures in Section~\ref{sec:eve/odd}. Inspired by~\cite{MR3194659}, in section~\ref{sec:herm}, we derive formulas for the structured eigenvalue backward error of $\lambda \in \C$ for Hermitian and related structures. 
The techniques of~\cite{MR3335496}  are extended for deriving the backward error formulas for RMPs with palindromic structure in Section~\ref{sec:pal}. The eigenvalues of a structured RMP behave differently under structured preserving and arbitrary perturbations to its coefficient matrices; see Section~\ref{sec:numeric}.

\paragraph{Notations}:~ In the following, ${\rm Herm}(n)$, ${\rm SHerm}(n)$,  ${\rm Sym}(n)$, and ${\rm Ssym}(n)$ respectively denote the sets of $n \times n$ Hermitian, skew-Hermitian, symmetric and skew-symmetric matrices. For a Hermitian matrix $H$, $\lambda_{\max}(H)$, $\lambda_{\min}(H)$ and $\lambda_2(H)$ respectively denote the largest, the smallest, and the second-largest eigenvalues of $H$. The smallest and the second largest singular values of  a matrix $A$ are respectively denoted by $\sigma_{\min}(A)$ and $\sigma_2(A)$. By $\|\cdot\|$, we denote the spectral norm of a vector or a matrix. By $I_m$, we denote the identity matrix of size $m \times m$, and $i$ stands for the complex number $\sqrt{-1}$. The  Kronecker  product is represented by $\otimes$ and we refer to~\cite{HorJ85} for the standard properties of the Kronecker product. 

\section{Preliminaries}\label{sec:prelims}

In order to measure perturbations of rational matrix polynomials, we introduce a norm on $\left(\mathbb{C}^{n \times n}\right)^{d+k+1}$.
\begin{definition}
For a tuple of matrices $(\Delta A_{0}, \ldots, \Delta A_{d},\Delta E_1,\ldots, \Delta E_k) \in (\mathbb{C}^{n \times n})^{d+k+1}$, we define 
\[
\nrm{\left(\Delta A_{0}, \ldots, \Delta A_{d},\Delta E_1, \ldots, \Delta E_k\right)}:=\sqrt{\left\|\Delta A_{0}\right\|^{2}+\cdots+\left\|\Delta A_{d}\right\|^{2}+\left\|\Delta E_{1}\right\|^{2}+\cdots+\left\|\Delta E_{k}\right\|^{2}},
\]
where $\|\cdot\|$ stands for the spectral norm.
\end{definition}
Let $G(z)$ be an $n \times n$ RMP of \emph{degree} $d$ of the form~\eqref{rmatrix}, i.e., $G(z)=\sum_{p=0}^d z^pA_p + \sum_{j=1}^k w_j(z)E_j$, where $w_j(z)=\frac{s_j(z)}{q_j(z)}$.  The RMP $G(z)$ is said to be \emph{regular} if $\text{det}(G(z)) \not \equiv 0$; otherwise, it is called \emph{singular}.
The pair $(\lambda,x)$ is referred to as an eigenpair of $G(z)$ if $G(\lambda)x=0$. The roots of $q_i(z)$ are the \emph{poles} of $G(z)$ and  $G(z)$ is not defined on these values. Throughout the paper, we assume that 
\begin{equation}\label{assump}
\text{\textbf{Assumption}}:~\lambda \in \C~ \text{such\,that}~ w_j(\lambda)\neq 0 ~\text{for}\, j=1,\ldots,k~\text{and}~ \lambda~ \text{is not a pole of}~ G(z).
\end{equation}
For the classical perturbation analysis of $G(z)$, one needs to consider perturbations in the coefficient matrices $A_j$'s and $E_j$'s, and also to scalar functions $w_j$'s, i.e., consider the following perturbed rational matrix polynomial 
\begin{equation}
\sum_{p=0}^d z^p(A_p+\Delta A_p) + \sum_{j=1}^k \left(w_j(z)+\delta w_j(z)\right)\left(E_j+\Delta E_j\right),
\end{equation}
and then study how the eigenvalues change under the perturbations. Unfortunately, such perturbation analysis turns out to be extremely difficult. Instead, one may assume in many applications that the perturbation in the scalar functions $w_j(z)$ are either known or can be bounded~\cite{MR3568297,Ahm19}.

 Motivated from the perturbation analysis in~\cite{MR3568297}, we consider the structure-preserving perturbations $\Delta G(z)$ to the RMP $G(z)$ of the form
\begin{equation}\label{eq:assumpert}
\Delta G(z)=\sum_{p=0}^d z^p\Delta A_p+\sum_{j=1}^kw_j(z)\Delta E_j.
\end{equation}
We are now in a position to define backward error for $G(z)$ of a $\lambda \in \C$.
\begin{definition}
    Let $G(\lambda)$ be a RMP of the form~\eqref{rmatrix} and $\lambda \in \mathbb{C}$ satisfying~\eqref{assump}. Further, let $\mathbb{S}\subseteq (\mathbb{C}^{n\times n})^{d+k+1}$. Then,
\begin{eqnarray}\label{def:error}
            \eta^{\mathbb{S}}(G,\lambda):=\inf\Bigg\{&\hspace{-.3cm} \nrm{(\Delta A_{0}, \ldots, \Delta A_{d},\Delta E_1, \ldots, \Delta E_k)} \; :\;(\Delta A_{0}, \ldots, \Delta A_{d},\Delta E_1, \ldots, \Delta E_k)\in \mathbb{S},\nonumber\\
             & {\rm det}\bigg(\sum_{p=0}^d \lambda^p(A_p-\Delta A_p)+\sum_{j=1}^kw_j(\lambda)(E_j-\Delta E_j)\bigg)=0 \Bigg\}
\end{eqnarray}
    is called the \emph{structured eigenvalue backward error} of $\lambda$ with respect to $G$ and $\mathbb{S}$.
\end{definition}
When $\mathbb{S}=(\mathbb{C}^{n\times n})^{d+k+1}$ in~\eqref{def:error}, it is called the \emph{unstructured eigenvalue backward error} and is denoted by $\eta(G,\lambda)$. Clearly, $\eta^{\mathbb{S}}(G,\lambda)=0$ if $G(z)$ is singular or $\lambda$ is an eigenvalue of $G$. Hence, in the following, it is reasonable to assume  that $G(z)$ is regular and that $G(\lambda)$ is non-singular (i.e., $\lambda$ is not an eigenvalue of $G(z)$). We also have that 
    \[
    \eta^{\mathbb{S}}(G,\lambda)\leq \nrm{(A_{0}, \ldots, A_{d},E_1, \ldots, E_k)} < \infty,
    \]
because the perturbation $(A_{0}, \ldots, A_{d},E_1, \ldots, E_k)$ makes the perturbed rational matrix as zero matrix.

As mentioned earlier, we extend the framework suggested in~\cite{MR3194659} and~\cite{MR3335496} to compute the structured eigenvalue backward error formulas for $G(z)$. For this, the following lemma is very crucial that converts the determinant condition in~\eqref{def:error} into some structured mapping conditions. This result is a generalization of~\cite[Lemma 2.4]{MR3194659}.
\begin{lemma}\label{determinant}
    Let $G(\lambda)$ be a RMP of the form~\eqref{rmatrix}, $\lambda \in \mathbb{C}$ such that $M:=G(\lambda)^{-1}$ exists, and let $\Delta G(z)$ be as defined in~\eqref{eq:assumpert}.  Then the following are equivalent.
    \begin{enumerate}
        \item[{\rm (i)}]~${\rm det}(G(\lambda)-\Delta G(\lambda))=0$
        \item[{\rm (ii)}]~there exist vectors $v_{A_p}$ for $p=0,\ldots, d$ and vectors $v_{E_j}$ for $j=1,\ldots, k$ satisfying $v_{\lambda}:=\sum_{p=0}^d\lambda^pv_{A_p}+\sum_{j=1}^kw_j(\lambda)v_{E_j} \neq 0$ such that  $v_{A_p}=\Delta A_pMv_{\lambda}$ for $p=0,\ldots,d$, and $v_{E_j}=\Delta E_jMv_{\lambda}$ for $j=1,\ldots, k$. 
    \end{enumerate}
\end{lemma}
\begin{proof}
    Let $\tilde{G}(\lambda):=G(\lambda)-\Delta G(\lambda)=\sum_{p=0}^d \lambda^i(A_p-\Delta A_p)+\sum_{j=1}^kw_j(\lambda)(E_j-\Delta E_j)$.\\
    (i)$\implies$(ii) First suppose that (i) holds, then the determinant condition of (i) implies that $\exists \; x \neq 0$ such that $\tilde{G}(\lambda)x=0$. Define $v_{A_p}=\Delta A_px$ for all $p=0,\ldots,d$ and $v_{E_j}=\Delta E_jx$ for all $j=1,\ldots, k$, then
    \begin{equation}\label{ab}
    	G(\lambda)x=G(\lambda)x-\tilde{G}(\lambda)x=\sum_{p=0}^d\lambda^pv_{A_p}+\sum_{j=1}^kw_j(\lambda)v_{E_j}=v_{\lambda}.
    \end{equation}
  Clearly, $v_{\lambda}\neq 0$ as $G(\lambda)$ is nonsingular. On multiplying~\eqref{ab} from the left with $\Delta A_pM$ and $\Delta E_jM$ we obtain
\begin{equation}
v_{A_p}  =\Delta A_iMv_{\lambda}~ \text{for}~ p=0,\ldots,d \quad \text{and}\quad 
            v_{E_j}  =\Delta E_jMv_{\lambda}~\text{for}~ j=1,\ldots,k.
\end{equation}

(ii)$\implies$(i) Next, suppose that (ii) holds. Then 
    \[
    \left(G(\lambda)-\Delta G(\lambda)\right)Mv_{\lambda}=v_{\lambda}-\Delta G(\lambda)Mv_{\lambda}=v_{\lambda}-v_{\lambda}=0.
    \]
    This implies that ${\rm det}(G(\lambda)-\Delta G(\lambda))=0$, since 
$Mv_{\lambda}\neq 0$.
\end{proof} 
Using Lemma~\ref{determinant} in~\eqref{def:error}, we obtain the following corollary.
\begin{corollary}\label{cor1}
 Let $G(\lambda)$ be a RMP of the form~\eqref{rmatrix} and let $\lambda \in \mathbb{C}$ such that $M:=G(\lambda)^{-1}$ exists.  Suppose that $\mathbb{S}\subseteq (\mathbb{C}^{n\times n})^{d+k+1}$. Then
 \begin{eqnarray}\label{strucerror}
 \eta^{\mathbb{S}}(G,\lambda):& \hspace{-.3cm}=\inf\Big\{  \left\|(\Delta A_{0}, \ldots, \Delta A_{d},\Delta E_1, \ldots, \Delta E_k)\right\| \; 
 :\;(\Delta A_{0}, \ldots, \Delta A_{d},\Delta E_1, \ldots, \Delta E_k)\in \mathbb{S}, \nonumber\\
            &\exists\, v_{A_0}, \ldots,v_{A_d}, v_{E_1},\ldots,v_{E_k} \in \C^n, ~
            \; v_{\lambda}=\sum_{p=0}^d\lambda^pv_{A_p}+\sum_{j=1}^kw_j(\lambda)v_{E_j}\neq 0,\nonumber\\
            & v_{A_p}=\Delta A_p Mv_{\lambda},\,p=0,\ldots,d,\;{\rm and}~\, v_{E_j}=\Delta E_jMv_{\lambda},\,j=1,\ldots,k\Big\}.
 \end{eqnarray}
\end{corollary}

Corollary~\ref{cor1} shows that $ \eta^{\mathbb{S}}(G,\lambda)$ can be simplified further using the minimal norm solutions to the structured mapping problems, where the structure of the mappings depends on the structure of $G(z)$. We will recall the structured mapping results from the literature whenever needed in the following sections.

The following lemma will be helpful to tackle the eigenvalue backward errors for the polynomials $G(z)$ with T-structures, like, skew-symmetric, T-even, T-odd, and T-palindromic. 

%
\begin{lemma}{\rm \cite[Lemma 2.4]{MR4404572}}\label{lem:B}
Let $F=\Omega_{m}^{T} \Lambda_{m} \otimes M +\Lambda_{m}^{T} \Omega_{m} \otimes M^{T}$, where the row vectors $\Omega_{m}, \Lambda_{m} \in \mathbb{C}^{1 \times m}$ are linearly independent and $M \in \mathbb{C}^{n \times n}$ is nonsingular. Then ${\rm rank}(F)=2n$.
\end{lemma}
%

As it was done in~\cite{MR3194659} for the eigenvalue backward errors of matrix polynomials, our focus will be to first reformulate the backward error $\eta^{\mathbb S}(G,\lambda)$ for various structures on $G(z)$ to either of the following two optimization problems:
\begin{enumerate}
\item[(i)] maximizing Rayleigh quotient of a Hermtian matrix with respect to Hermitian constraints, i.e., a problem of the form 
\begin{equation}\label{opt:herm}
\sup \left\{\frac{v^*Jv}{v^*v} ~:~ v \in \C^{n} \setminus \{0\}, \; v^*H_pv=0,\; p=0,\ldots, r \right\}
\end{equation}
for some $n \times n$ Hermitian matrices $J, H_0, \ldots,H_r$, or,
\item[(ii)]  maximizing Rayleigh quotient of a Hermtian matrix with respect to symmetric constraints, i.e., a problem of the form 
\begin{equation}\label{opt:sym}
m_{js_0\ldots s_r}(J,S_0,\ldots,S_r)=\sup\left\{\frac{v^*Hv}{v^*v} ~:~ v \in \C^{n} \setminus \{0\}, \; v^TS_pv=0, \; p=0,\ldots,r \right\}
\end{equation}
for some $n \times n$ Hermitian matrix $H$ and some $n \times n$ symmetric matrices $S_0,\ldots,S_r$.
\end{enumerate}

The two problems in~\eqref{opt:herm} and~\eqref{opt:sym} of maximizing 
Rayleigh quotient were studied in~\cite{MR3194659} for constraints involving Hermitian matrices and in~\cite{MR4404572} for constraints involving symmetric matrices. We restate these results here in the form that will allow us to use them directly to estimate eigenvalue backward error $\eta^{\mathbb S}(G, \lambda)$ for various structures. 

\begin{theorem}{\rm \cite{MR3194659}}\label{theorem:optimize2}
	Let $J, H_{0}, \ldots, H_{r} \in \mathbb{C}^{n \times n}$ be Hermitian matrices. Assume that any nonzero linear combination $\alpha_{0} H_{0}+\cdots+\alpha_{r} H_{r},\left(\alpha_{0}, \ldots, \alpha_{r}\right) \in \mathbb{R}^{r+1} \setminus\{0\}$ is indefinite (i.e., strictly not semidefinite). Then the following statements hold:
\begin{enumerate}
\item[(i)] The function $L: \mathbb{R}^{r+1} \rightarrow \mathbb{R},\left(t_{0}, \ldots, t_{r}\right) \rightarrow\lambda_{\max }\left(J+t_{0} H_{0}+\cdots+t_{r} H_{r}\right)$ is convex and has a global minimum
\[
	\hat \lambda_{\max }:=\min _{t_{0}, \ldots, t_{r} \in \mathbb{R}} L\left(t_{0}, \ldots, t_{r}\right).
\]
Further, we have 
\begin{equation}\label{eq:hermineqres}
\sup \left\{\frac{u^{*} J u}{u^{*} u} ~:~ u \in \mathbb{C}^{n} \backslash\{0\}, u^{*} H_{j} u=0, j=0, \ldots, r\right\}\leq \hat\lambda_{\max }.
\end{equation}
\item[(ii)] If the minimum $\hat \lambda_{\max }$ of $L$ 
is attained at $\left(\hat t_{0}, \ldots,\hat t_{r}\right) \in \mathbb{R}^{r+1}$ and is a simple 
 eigenvalue of $\hat H:=J+\hat t_{0}H_{0}+\cdots+\hat t_{r} H_{r}$, then there exists an eigenvector $u \in \mathbb{C}^{n} \backslash\{0\}$ of $\hat H$ associated with $\hat \lambda_{\max }$ satisfying
\[
	u^{*} H_{j} u=0 \quad {\rm for }~ j=0, \ldots, r.
\]
In this case, equality holds in~\eqref{eq:hermineqres}. 
\end{enumerate}
\end{theorem}

Next, we state a result from~\cite{MR4404572} to estimate~\eqref{opt:sym}. 
\begin{theorem}{\rm \cite{MR4404572}}\label{theorem:optimize1}
	Let $J \in {\rm Herm}(n)$ and $S_0, \ldots,S_r\in {\rm Sym}(n)$,
	let $\psi: \R^{2r+1}\mapsto \R$ be defined by $(t_0,\ldots,t_{2r})\mapsto \lambda_2(F(t_0,\ldots,t_{2r}))$, where 
\begin{equation}\label{eq:defF}
F(t_0,\ldots,t_{2r})=\begin{bmatrix}
J & \overline{f(t_0,\ldots,t_{2r})}\\
f(t_0,\ldots,t_{2r}) & \overline{J}
\end{bmatrix},
\end{equation}
where $f(t_0,\ldots,t_{2r}):=(t_0+it_1)S_0+\cdots+(t_{2r-2}+it_{2r-1})S_{r-1}+t_{2r}S_r$.
	Then,
	\begin{equation}\label{eq:summary1}
        	m_{hs_0\ldots s_r}(J,S_0,\ldots,S_r)  \leq \inf_{t_0,\ldots,t_{2r}\in\R}\psi(t_0,\ldots,t_{2r}).
	\end{equation}
	Moreover, the following statements hold.
	\begin{enumerate}
    	\item If ${\rm rank}(f(t_0,\ldots,t_{2r}))\geq 2$ for all $(t_0,\ldots,t_{2r}))\in \R^{2r+1} \setminus \{0\}$, then infimum in~\eqref{eq:summary1} is attained in the region $t_0^2+\cdots + t_{2r}^2\leq \beta^2$, where
   \[
   \beta=\frac{\lambda_{\max}(J)-\lambda_{\min}(J)}{c}\quad  \text{and} \quad c=\min\big\{\sigma_2(f(t_0,\ldots,t_{2r})) :~  t_0^2+\cdots + t_{2r}^2=1\big\}.
   \]
%
\item Suppose that the infimum in~\eqref{eq:summary1} is attained at $(\hat t_0,\ldots,\hat t_{2r})$.  If $\psi(\hat t_0,\ldots,\hat t_{2r})$ is a simple eigenvalue of $F(\hat t_0,\ldots,\hat t_{2r})$, then equality holds in~\eqref{eq:summary1}.
	\end{enumerate}
\end{theorem}

We close this section by stating a result that gives an explicit formula for the unstructured eigenvalue backward error $\eta(G, \lambda)$, proof of which is similar to~\cite[Theorem 4.1]{MR3194659} and hence given in~\ref{app:sec1} for future reference. 

\begin{theorem}\label{thm:unstrbacerrfor}
Let $G(\lambda)$ be a RMP of the form~\eqref{rmatrix} and let $\lambda \in \mathbb{C}$ satisfying~\eqref{assump}. Then, 
\begin{equation}\label{error:unstruct}
\eta(G,\lambda)= \frac{\sigma_{\min} (G(\lambda))}{\|(1,\lambda,\ldots,\lambda^d,w_1(\lambda),\ldots,w_k(\lambda))\|},
\end{equation}
where $\sigma_{\min}(A)$ stands for the smallest singular value of a matrix   $A$.
\end{theorem}

\section{Symmetric and skew-symmetric structures} \label{sec:sym}

In this section, we consider  symmetric or skew-symmetric RMPs $G(z)$ in the form~\eqref{rmatrix} which satisfy $(G(z))^T=G(z)$ or $(G(z))^T=-G(z)$, respectively. Consequently, the coefficient matrices $A_i$'s and $E_i$'s are all symmetric when $G(z)$ is symmetric and skew-symmetric when $G(z)$ is skew-symmetric, see Table~\ref{tab:my_label}. The corresponding structured eigenvalue backward error is denoted by $\eta^{{\rm sym}}(G,\lambda)$ when $\mathbb S ={\rm Sym(n)}^{d+k+1}$ and $\eta^{{\rm ssym}}(G,\lambda)$ when $\mathbb S ={\rm Ssym(n)}^{d+k+1}$ in~\eqref{def:error}.

Suppose $G(z)$ is a skew-symmetric RMP of the form~\eqref{rmatrix}   and let $\lambda \in \mathbb{C}$ satisfy~\eqref{assump}. Suppose that $G(\lambda)$ is nonsingular.  Then from Corollary~\ref{cor1}, computing the backward error $\eta^{{\rm ssym}}(G,\lambda)$ involves solving skew-symmetric mapping problems. In fact, for each $p=0,\ldots,d$ and for each $j=1,\ldots,k$, we need to find skew-symmetric matrices $\Delta A_p$ and $\Delta E_j$ such that 
\begin{equation}\label{eq:skew-symconds}
\Delta A_pMv_{\lambda}=v_{A_p} \quad \text{and} \quad \Delta E_jMv_{\lambda} =v_{E_j}.
\end{equation}
In view of~\cite[Theorem 2.2.2]{Adh08}, for given vectors $Mv_{\lambda}$, $v_{A_p} \in \C^n$ and $Mv_{\lambda}$, $v_{E_j} \in \C^n$, there exist skew-symmetric matrices $\Delta A_p$ and $\Delta E_j$ such that~\eqref{eq:skew-symconds} holds if and only if $(Mv_{\lambda})^Tv_{A_p}=0$ and $(Mv_{\lambda})^Tv_{E_j}=0$. Further, the minimum spectral norm of such mappings satisfy
\[
\|\Delta A_p\|=\frac{\|v_{A_p}\|}{\|Mv_{\lambda}\|} \quad \text{and}\quad \|\Delta E_j\|=\frac{\|v_{E_j}\|} 
  {\|Mv_{\lambda}\|}.
\]
Thus using the minimal norm skew-symmetric mappings in Corolllary~\ref{cor1}, the norm of the perturbation $(\Delta A_{0}, \ldots, \Delta A_{d},\Delta E_1, \ldots, \Delta E_k)\in {\rm Ssym(n)}^{d+k+1}$ satisfies 
%
%
\begin{eqnarray}\label{norm}
		\nrm{(\Delta A_{0}, \ldots, \Delta A_{d},\Delta E_1, \ldots, \Delta E_k)}^2 & =&\frac{\|v_{A_0}\|^2}{\|Mv_{\lambda}\|^2}+\cdots+\frac{\|v_{A_d}\|^2}{\|Mv_{\lambda}\|^2}+\frac{\|v_{E_1}\|^2}{\|Mv_{\lambda}\|^2}+\cdots + \frac{\|v_{E_k}\|^2}{\|Mv_{\lambda}\|^2} \nonumber\\
		& =& \frac{\|v_{A_0}\|^2+\cdots+\|v_{A_d}\|^2+\|v_{E_1}\|^2+\cdots+\|v_{E_k}\|^2}{\|Mv_{\lambda}\|^2} \nonumber\\
		& =& \frac{u^*u}{u^*Ju},
\end{eqnarray}

which is obtained by taking
\begin{equation}\label{skew:matJ}
u = \begin{bmatrix}
v_{A_0}^T & \cdots & v_{A_d}^T & v_{E_{1}}^T & \cdots & v_{E_{k}}^T\end{bmatrix}^T, \quad \text{and}\quad
J = \Lambda\Lambda^* \otimes M^*M,
\end{equation}
where $\Lambda:=[1,\lambda,\ldots,\lambda^d, w_1(\lambda),\ldots,w_k(\lambda)]^T$.

The conditions $(Mv_{\lambda})^Tv_{A_p}=0$ and $(Mv_{\lambda})^Tv_{E_j}=0$ for each $p=0,\ldots,d$ and $j=1,\ldots,k$ can be reduced to $d+k+1$ conditions $u^TC_{p}u=0$ with 
\begin{equation}\label{eq:skewreformCi}
	C_{p}=\Lambda^Te_{p+1}^T \otimes M^T, \quad p=0,\ldots,d+k,
\end{equation}
and it is easy to verify that 
\begin{equation}\label{eq:skewreformvlam}
v_\lambda \neq 0 \quad \Leftrightarrow \quad u^*Ju \neq 0,
\end{equation}
since $M$ is nonsingular.
Further note that if we define $S_p=C_{p}+C_{p}^T$, then for each $p=0,\ldots,d+k$, 
\begin{eqnarray}\label{eq:skewreformSi}
    u^TCu=0 \quad \Leftrightarrow \quad u^T(C^T+C)u=0\quad 
     \Leftrightarrow \quad u^TS_pu=0.
\end{eqnarray}
Using the above observation in Corollary~\ref{cor1}, we get
 \begin{eqnarray}
 (\eta^{{\rm ssym}}(G,\lambda))^2& =&\inf\bigg\{  \sum_{p=0}^d\frac{\|v_{A_p}\|^2}{\|Mv_{\lambda}\|^2}+\sum_{j=1}^k\frac{\|v_{E_j}\|^2}{\|Mv_{\lambda}\|^2} \; 
 :\;\exists\, v_{A_0}, \ldots,v_{A_d}, v_{E_1},\ldots,v_{E_k} \in \C^n \nonumber\\
            &&
            \; ~v_{\lambda}\neq 0,~ (Mv_{\lambda})^Tv_{A_p}=0,\,p=0,\ldots,d,\;\text{and}~\, (Mv_{\lambda})^Tv_{E_j}=0,\,j=1,\ldots,k\bigg\} \nonumber \\
            & =&\inf\left\{\frac{u^*u}{u^*Ju}\;: \; u \in \C^{(d+k+1)n},~ u^*Ju \neq 0, u^TS_{p}u=0 \quad p=0,\ldots,d+k\right\} \label{eq:skewreform1} \\
                        & =&\inf\left\{\frac{u^*u}{u^*Ju}\;: \; u \in \C^{(d+k+1)n}\setminus\{0\},~ u^TS_{p}u=0 \quad p=0,\ldots,d+k\right\} \label{eq:skewreform2} \\
& = & (m_{js_0\ldots s_{d+k}}(J,S_0,S_1,\ldots, S_{d+k}))^{-1}, \label{eq:skewreform3}
 \end{eqnarray}
 where the equality in~\eqref{eq:skewreform1} follows due to~\eqref{norm}, ~\eqref{eq:skewreformSi}, and~\eqref{eq:skewreformvlam}. The equality in~\eqref{eq:skewreform2}  follows since the infimum in~\eqref{eq:skewreform1} is not attained at the vectors $u$ for which $u^*Ju=0$ since $\eta^{{\rm ssym}}(G,\lambda)$  is finite. Finally~\eqref{eq:skewreform3} follow due to~\eqref{opt:sym}.
 
 Thus a direct application of Theorem~\ref{theorem:optimize1} in~\eqref{eq:skewreform3} yields the result that computes a lower bound for the structured eigenvalue backward error $\eta^{{\rm ssym}}(G,\lambda)$.

\begin{theorem}\label{thm:formulaSkew}
Let $G(z)$ be a skew-symmetric RMP of the form~\eqref{rmatrix}  and let $\lambda \in \mathbb{C}$ satisfying~\eqref{assump}. Suppose that
 $M:=(G(\lambda))^{-1}$ exists and for $(t_0,\ldots,t_{2(d+k)}) \in \R^{2(d+k)+1}$ define $f(t_0,\ldots,t_{2(d+k)}):=(t_0+it_1)S_0+(t_2+it_3)S_1+\cdots+t_{2(d+k)}S_{d+k}$ for matrices $S_0,\ldots,S_{d+k}$ given by~\eqref{eq:skewreformSi}. Then
\[
\eta^{{\rm ssym}}(G,\lambda) \geq \left(\inf_{(t_0,\ldots,t_{2(d+k)})\in \R^{2(d+k)+1}} \lambda_2\left(\begin{bmatrix}
J & \overline{f(t_0,\ldots,t_{2(d+k)})}\\
f(t_0,\ldots,t_{2(d+k)}) & \overline{J}
\end{bmatrix}\right)\right)^{-\frac{1}{2}}.
\]
where $ J$ is defined by~\eqref{skew:matJ}.
\end{theorem}

Now suppose $G(z)$ is a symmetric RMP of the form~\eqref{rmatrix}. As it was shown in~\cite{Sha16} for symmetric matrix polynomials, it can be easily verified that there is no difference between the structured and the unstructured backward errors for symmetric RMPs. More precisely, we have the following result.

\begin{theorem}
Let $G(z)$ be a RMP of the form~\eqref{rmatrix} and let $\lambda \in \mathbb{C}$ satisfying~\eqref{assump}. Suppose that $M:=(G(\lambda))^{-1}$ exists.  Then 
	\begin{equation}
		\eta^{{\rm sym}}(G,\lambda)=\eta(G,\lambda),
	\end{equation} 
	where $\eta(G,\lambda)$ is given in Theorem~\ref{thm:unstrbacerrfor}.
\end{theorem}
\begin{proof}
In view of Corollary~\ref{cor1}, for any given vectors $v_{A_0},\ldots,v_{A_d},v_{E_1},\ldots,v_{E_k} \in \C^n$ such that $v_\lambda \neq 0$, symmetric matrices $\Delta A_p$, $p=0,\ldots,d$, and $\Delta E_j$, $j=1,\ldots,k$ may be chosen to satisfy 
\[
\Delta A_pMv_{\lambda}=v_{A_p} \quad \text{and} \quad \Delta E_jMv_{\lambda} =v_{E_j}
\]
without any restrictions on $Mv_\lambda$, $v_{A_p}$, or $v_{E_j}$. 
Thus using the minimal norm symmetric mappings from~~\cite[Theorem 2.2.2]{Adh08} in Corolllary~\ref{cor1}, and following the arguments similar to the proof of Theorem~\ref{thm:formulaSkew}, we have
\[
\eta^{{\rm sym}}(G,\lambda)=\frac{1}{\sqrt{\lambda_{\max}(J)}}=\frac{\sigma_{\min} (G(\lambda))}{\|(1,\lambda,\ldots,\lambda^d,w_1(\lambda),\ldots,w_k(\lambda))\|}=\eta(G,\lambda),
\]
where 
$J = \Lambda\Lambda^* \otimes M^*M$ and  $\Lambda:=[1,\lambda,\ldots,\lambda^d, w_1(\lambda),\ldots,w_k(\lambda)]^T$.
\end{proof}

\section{T-even/T-odd}\label{sec:eve/odd}

In this section, we consider T-even and T-odd RMP's of the form~\eqref{rmatrix} which satisfy $(G(z))^T=G(-z)$ or $(G(z))^T=-G(-z)$, respectively. This implies that the coefficient matrices of T-even RMPs $G(z)$ satisfy that $A_p^T=(-1)^pA_p$ for $p=0,\ldots,d$, and for $j=1,\ldots,k$ $E_j^T=E_j$ if $w_j(z)$ is an even function and $E_j^T=-E_j$ if $w_j(z)$ is an odd function. Similarly, for T-odd RMPs $G(z)$, $A_p^T=(-1)^{p+1}A_p$ for $p=0,\ldots,d$, and for $j=1,\ldots,k$ $E_j^T=E_j$ if $w_j(z)$ is an odd function and $E_j^T=-E_j$ if $w_j(z)$ is an even function, see Table~\ref{tab:my_label}.

To tackle, the T-even/T-odd polynomials, we partitionn the set $X=\{1,\ldots,k\}$ into two subsets of indices $X_o:=\{p_1,\ldots,p_r\}$ and $X_e:=\{p_{r+1},\ldots,p_k\}$ with $p_i < p_j$ for $i<j$ such that 
$w_{p}(z)$ is an odd function if $p \in X_o$ and $w_{p}(z)$ is an even function if $p \in X_e$.
Thus for T-even structures, we have $A_p^T=(-1)^pA_p$ for $p=0,\ldots,d$, and for  $E_j^T=E_j$ if $j \in X_e$ and $E_j^T=-E_j$ if $j \in X_o$.

Let us define two subsets of matrix tuples corresponding to T-even and T-odd structures by
\begin{eqnarray}
 \mathcal{S}_e:&=\big \{ (A_0,\ldots, A_d, E_1,\ldots ,E_k) \in (\C^{n,n})^{d+k+1}\; : \; A_p^T=(-1)^pA_p, \; p=0,\ldots, d,\\
 & E_j^T=E_j\;\text{if}\; j\in X_e \;~ \text{and}\; ~E_j^T=-E_j \; \text{if}\;j \in X_o \big \},
\end{eqnarray}
and 
\begin{eqnarray}
 \mathcal{S}_o:&=\big \{ (A_0,\ldots, A_d, E_1,\ldots ,E_k) \in (\C^{n,n})^{d+k+1}\; : \; A_p^T=(-1)^{p+1}A_p, \; p=0,\ldots, d,\\
 & E_j^T=E_j\;\text{if}\; j\in X_o \;~ \text{and}\; ~E_j^T=-E_j \; \text{if}\;j \in X_e  \big \},
\end{eqnarray}
and denote the T-even eigenvalue backward error by $\eta^{{\rm even_T}}(G,\lambda)$ when $\mathbb S=\mathcal S_e$ in~\eqref{def:error}, and the T-odd eigenvalue backward error by $\eta^{{\rm odd_T}}(G,\lambda)$ when $\mathbb S=\mathcal S_o$ in~\eqref{def:error}. In the rest of this section, we derive a result to estimate the T-even backward error $\eta^{{\rm even_T}}(G,\lambda)$. A similar result for $\eta^{{\rm odd_T}}(G,\lambda)$ can be obtained analogously. 

In view of Corollary~\ref{cor1}, when $\mathbb S=\mathcal S_e$, we have 
%
\begin{eqnarray}\label{error:ham}
        \eta^{{\rm even_T}}(G,\lambda)=&\inf\Big\{\nrm{(\Delta A_0,\ldots,\Delta E_k)} \; : \; (\Delta A_0,\ldots, \Delta A_d, \Delta E_1,\ldots,\Delta E_k)\in \mathcal{S}_e,\nonumber \\
        &\; \exists\, v_{A_0}, \ldots,v_{A_d}, v_{E_1},\ldots,v_{E_k} \in \C^n \;v_{\lambda} \neq 0 \; \Delta A_pMv_{\lambda}=v_{A_p}, \nonumber \\ 
        & ~p=0,\ldots, d,\text{and}~ \Delta E_jMv_{\lambda}=v_{E_j},  \; j=1,\ldots,k \Big \}.
\end{eqnarray}
For any given vectors $x,y \in \C^n$, we can always find a symmetric matrix $\Delta$ such that $\Delta x=y$ without any restriction on the vectors $x$ and $y$, and the minimal spectral norm of such a $\Delta$ is given by $\frac{\|y\|}{\|x\|}$~\cite[Theorem 2.2.2]{Adh08}. On the other hand, for given vectors $x,y \in \C^n$, there exists a skew-symmetric matrix $\Delta$ mapping $x$ to $y$ if and only if $y^Tx=0$, and the minimal spectral norm of such a $\Delta$ satisfy $\frac{\|y\|}{\|x\|}$~\cite[Theorem 2.2.2]{Adh08}. 

Thus any given pairs $(Mv_{\lambda},v_{A_p})$ and $(Mv_{\lambda},v_{E_j})$, we can always find symmetric matrices $\Delta A_p$ and $ \Delta E_j$ such that $\Delta A_pMv_{\lambda}=v_{A_p}$ and $\Delta E_jMv_{\lambda}=v_{E_j}$ without any restrictions on the vectors $Mv_{\lambda}$, $ v_{A_p}$, and $v_{E_j}$. 
On the other hand, there exists skew-symmetric matrices $\Delta A_p$, and $\Delta E_j$ satisfying $\Delta A_pMv_{\lambda}=v_{A_p}$ and $\Delta E_jMv_{\lambda}=v_{E_j}$ if and only if $v_{A_p}^TMv_{\lambda}=0$ and $v_{E_j}^TMv_{\lambda}=0$, which is true if and only if $u^TC_ju=0$, where 
\begin{equation}\label{Ham:Lambda}
    \begin{split}
        C_j & = \begin{cases}
                    \Lambda^Te_{2j+2}^T\otimes M^T \quad & \text{for} \; j=0,\ldots,m\\
\Lambda^Te_{d+1+p_{j-m}}^T\otimes M^T \quad & \text{for} \; j=m+1,\ldots,m+r,
                \end{cases}\\
    \end{split},
\end{equation}
where $m:=\lfloor\frac{d-1}{2}\rfloor $, $\Lambda  =[1,\lambda,\ldots,\lambda^d,w_{1}(\lambda),\ldots,w_{k}(\lambda)]$, and $u = \begin{bmatrix}
v_{A_0}^T & \cdots & v_{A_d}^T & v_{E_{1}}^T & \cdots & v_{E_{k}}^T\end{bmatrix}^T$.

Also using minimal norm symmetric and skew-symmetric mappings from~\cite[Theorem 2.2.2]{Adh08}, the norm of the perturbation $(\Delta A_0,\ldots,\Delta A_d,\Delta E_1,\ldots,\Delta E_k)$ satisfies that 
\begin{equation}\label{norm:Hamiltonian}
    \nrm{(\Delta A_0,\ldots,\Delta A_d,\Delta E_1,\ldots,\Delta E_k)}^2  =\sum_{p=0}^d\frac{\|v_{A_p}\|^2}{\|Mv_{\lambda}\|^2}+\sum_{j=1}^k\frac{\|v_{E_j}\|^2}{\|Mv_{\lambda}\|^2} = \frac{u^*u}{u^*Ju},
\end{equation}
where $J = \Lambda^{*}\Lambda \otimes M^*M$. Similar to Section~\ref{sec:sym}, we have that  for $j=0,\ldots, m+r$, 
\begin{equation}\label{matrix:symham}
   u^TC_ju=0 \iff u^TS_ju \quad \text{and}\quad   v_\lambda \neq 0  \iff  u^*Ju \neq 0,
\end{equation}
where  $S_j=C_j+C_j^T$. In view of~\eqref{norm:Hamiltonian},~ \eqref{matrix:symham} and~\eqref{error:ham},  we have 
\begin{eqnarray}\label{eq:Teven_reform}
(\eta^{{\rm even_T}}(G,\lambda))^2&=&\inf\left\{ \frac{u^*u}{u^*Ju} \;:\; u \in \C^{(d+k+1)n},~u^*Ju \neq 0,\; u^TS_ju=0,~j=0,\ldots, m+r\right\} \nonumber \\
&=& \inf\left\{ \frac{u^*u}{u^*Ju} \;:\; u \in \C^{(d+k+1)n}\setminus\{0\},\; u^TS_ju=0,~j=0,\ldots, m+r\right\} \nonumber \\
& = & (m_{js_0\ldots s_{m+r}}(J,S_0,S_1,\ldots, S_{m+r}))^{-1},
\end{eqnarray}
where the second last inequality is due to the fact that $\eta^{{\rm even_T}}(G,\lambda)$ is finite and the infimum is not attained at the vectors $u$ for which $u^*Ju=0$. The last equality in~\eqref{eq:Teven_reform} is due to~\eqref{opt:sym}.

Since, we aim to apply Theorem~\ref{theorem:optimize1} in~\eqref{eq:Teven_reform}, we need to check that the matrices $S_0, \ldots, S_{m+r}$ satisfy the rank constraint, i.e., for any real tuple $(t_0,\ldots,t_{2(m+r)})$, $\text{rank}((t_0+it_1)S_0+\cdots + t_{2(m+r)}S_{m+r})\geq 2$. This can be verified as
\begin{equation}
        \text{rank}((t_0+it_1)S_0+\cdots + t_{2(m+r)}S_{m+r})  = \text{rank}(\Lambda^T\Omega \otimes M^T + \Omega^T\Lambda \otimes M),
\end{equation}
where $\Omega \in \mathbb{C}^{1 \times (d+k+1)}$ is defined by 
\begin{equation}
\Omega:=
    \begin{cases}
        [0,(t_0+it_1),\cdots,0,(t_{2m}+it_{2m+1}),(t_{2m+2}+it_{2m+3}), \cdots, t_{2(m+r)}, 0, \cdots, 0] & \text{if d odd}\\
        [0,(t_0+it_1),\cdots,0,(t_{2m}+it_{2m+1}),0,(t_{2m+2}+it_{2m+3}), \cdots, t_{2(m+r)}, 0, \cdots, 0] & \text{if d even.}
    \end{cases}
\end{equation}

Observe that since $\Omega$ and $\Lambda$ are linearly independent, from Lemma~\ref{Gamma} we have that 
\[
\text{rank}(\Lambda^T\Omega \otimes M^T + \Omega^T\Lambda \otimes M) =2n \geq 2.
\]
Thus as a direct application of Theorem~\ref{theorem:optimize1} in~\eqref{eq:Teven_reform}, we obtain the following result for the T-even backward error $\eta^{{\rm even_T}}(G,\lambda)$.

\begin{theorem}\label{thm:formulaTeven}
Let $G(z)$ be a T-even RMP of the form~\eqref{rmatrix}  and let $\lambda \in \mathbb{C} \setminus\{0\}$ satisfying~\eqref{assump}. Suppose that $M:=(G(\lambda))^{-1}$ exists. 
Define 
 $m:=\left \lfloor{\frac{d-1}{2}}\right \rfloor$ and let 
Let $\psi: \mathbb{R}^{2(m+r)+1}\rightarrow \mathbb{R}$ be defined by $(t_0,\ldots,t_{2(m+r)})\rightarrow \lambda_2(F(t_0,\ldots,t_{2(m+r)}))$, where $F(t_0,\ldots,t_{2(m+r)})$ is as defined in \eqref{eq:defF} for matrices $J,S_0,\ldots,S_{m+r}$ of~\eqref{matrix:symham}. Then,
\begin{itemize}
\item The function $\psi$ has a global minimum
\[
\widehat \lambda_2 := \min_{(t_0,\ldots,t_{2(m+r)})\in \R^{2(m+r)+1}} \psi (t_0,\ldots,t_{2(m+r)}) \quad \text{and}\quad 
\eta^{{\rm even_T}}(G,\lambda) \geq \frac{1}{\sqrt{\widehat \lambda_2}}.
\]
\item Moreover, if the minimum $\widehat \lambda_2$ of $\psi$  is attained 
at $(\hat t_0,\ldots,\hat t_{2(m+r)}) \in \R^{2(m+r)+1}$ and is a simple eigenvalue of $F(\hat t_0,\ldots,\hat t_{2(m+r)}) $, then 
\[
\eta^{{\rm even_T}}(G,\lambda) = \frac{1}{\sqrt{\widehat \lambda_2}}.
\]
\end{itemize}
\end{theorem}

By following the lines of the T-even RMPs, the eigenvalue backward error for T-odd RMPs can also be obtained. More precisely, we have the following analogue of~\eqref{eq:Teven_reform} for T-odd structure,
\begin{equation}\label{reduced:Todd}
(\eta^{{\rm odd_T}}(G,\lambda))^2=\inf\left\{ \frac{u^*u}{u^*Ju} \;:\;u \in \C^{(d+k+1)n},\; u^TS_ju=0 \; \text{for} \; j=0,\ldots,m+d-r\right\},
\end{equation}
where $m=\lfloor \frac{d}{2}\rfloor$, $J$ is as given in~\eqref{matrix:symham}, and for $j=0,\ldots,m+d-r$ $S_j=C_j+C_j^T$ with $C_j$ defined as 
\begin{equation*}
    \begin{split}
        C_j & = \begin{cases}
                    \Lambda^Te_{2j+1}^T\otimes M^T \quad & \text{for} \; j=0,\ldots,m,\\
  \Lambda^Te_{d+1+p_{j-m+r}}^T\otimes M^T  \quad & \text{for} \; j=m+1,\ldots,m+d-r,
                \end{cases}.\\
    \end{split}
\end{equation*}
 Thus a direct application of Theorem~\ref{theorem:optimize1} in~\eqref{reduced:Todd}, we have the following result for $\eta^{{\rm odd_T}}(G,\lambda)$.

\begin{theorem}\label{thm:formulaTodd}
Let $G(z)$ be a T-odd RMP of the form~\eqref{rmatrix}  and let $\lambda \in \mathbb{C} \setminus\{0\}$ satisfying~\eqref{assump}. Suppose that $M:=(G(\lambda))^{-1}$ exists.  Define 
 $m:=\left \lfloor{\frac{d}{2}}\right \rfloor$ and let 
Let $\psi: \mathbb{R}^{2(m+d-r)+1}\rightarrow \mathbb{R}$ be defined by $(t_0,\ldots,t_{2(m+d-r)})\rightarrow \lambda_2(F(t_0,\ldots,t_{2(m+d-r)}))$, where $F(t_0,\ldots,t_{2(m+r)})$ is as defined in \eqref{eq:defF} for matrices $J,S_0,\ldots,S_{m+d-r}$ of~\eqref{reduced:Todd}. Then,
\begin{itemize}
\item The function $\psi$ has a global minimum
\[
\widehat \lambda_2 := \min_{(t_0,\ldots,t_{2(m+d-r)})\in \R^{2(m+d-r)+1}} \psi (t_0,\ldots,t_{2(m+d-r)}) \quad \text{and}\quad 
\eta^{{\rm even_T}}(G,\lambda) \geq \frac{1}{\sqrt{\widehat \lambda_2}}.
\]
\item Moreover, if the minimum $\widehat \lambda_2$ of $\psi$  is attained 
at $(\hat t_0,\ldots,\hat t_{2(m+d-r)}) \in \R^{2(m+d-r)+1}$ and is a simple eigenvalue of $F(\hat t_0,\ldots,\hat t_{2(m+d-r)}) $, then 
\[
\eta^{{\rm odd_T}}(G,\lambda) = \frac{1}{\sqrt{\widehat \lambda_2}}.
\]
\end{itemize}
\end{theorem}

%
\section{Hermitian and related strucuture} \label{sec:herm}

In this section, we consider RMPs $G(z)$ of the form~\eqref{rmatrix} with Hermitian and related structures like skew-Hermitian, $*$-even, and $*$-odd. Recall, from Table~\ref{tab:my_label} that $G(z)$ is Hermitian if and only if the coefficient matrices $A_j$'s and $E_j$'s in  $G(z)$ are all Hermitian and the weight functions satisfy that $(w_j(z))^*=w_j(\overline z)$ for $j=1,\ldots,k$, i.e., $(G(z))^*=G(\overline z)$.  In this case, the corresponding structured eigenvalue backward error is denoted by $\eta^{{\rm Herm}}(G,\lambda)$ and defined by~\eqref{def:error} when $\mathbb S=\left({\rm Herm}(n)\right)^{d+k+1}$.

As shown in~\cite{MR3194659} for Hermitian matrix polynomials, when $\lambda \in \R$, it is easy to check that there is no difference between the structured and the unstructured backward errors. More precisely, if $G(z)$ is a Hermitian RMP of the form~\eqref{rmatrix} and $\lambda \in \R$. Then 
\begin{equation}\label{eq:unstrhermres}
\eta^{{\rm Herm}}(G, \lambda)=\eta(G, \lambda)=\frac{\sigma_{\min }(G(\lambda))}{\left\|\left(1, \lambda,\ldots,\lambda^d,w_1(\lambda) \ldots, w_{k}(\lambda)\right)\right\|} .
\end{equation}
This follows from the fact that when $\lambda$ is real, then the perturbation matrices $\Delta A_{p}$, $p=0,\ldots,d$ and $\Delta E_j$, $j=1,\ldots,k$ in~\eqref{matrix:unst} are all Hermitian, which yields~\eqref{eq:unstrhermres}.

However, this is not the case when $\lambda \in \C\setminus \R$ as shown in the following result. 

%

\begin{theorem}\label{thm:Herm}
Let $G(z)$ be a Hermitian RMP of the form~\eqref{rmatrix} and let $\lambda \in \C \setminus \R$ satisfying~\eqref{assump}. Suppose that $M:=(G(\lambda))^{-1}$ exists. Define 
$\Lambda=\big[
1 ~\lambda ~ \cdots ~ \lambda^d ~ w_1(\lambda) ~ \cdots ~ w_k(\lambda) \big]$ 
and set,
\begin{equation*}
J := \Lambda^* \Lambda \otimes M^*M \quad \text{and}\quad 
H_j := i\left(\Lambda^*e_{j+1} \otimes M^* - e_{j+1}\Lambda \otimes M\right) \quad j=0,\ldots,d+k,
\end{equation*}
where $e_{j+1}$ denotes the $(j+1)^{th}$ standard unit vector of 
$\mathbb R^{d+k+1}$. Then, 
\[
 \widehat \lambda_{\max}:= \min_{t_0,\ldots,t_{d+k}\in \mathbb R} \lambda_{\max}\left( J+ t_0H_0 + \cdots + t_{d+k}H_{d+k} \right) 
 \]
is attained for some $(\widehat t_0,\ldots, \widehat t_{d+k})\in \mathbb R^{d+k+1}$, and we have $\eta^{{\rm Herm}}(G,\lambda) \geq  \frac{1}{\sqrt{\widehat \lambda_{\max}}}$.
 Furhter, if $\widehat \lambda_{\max}$ is a simple eigenvalue of 
$J+  \widehat t_0H_0 + \cdots + \widehat  t_{d+k}H_{d+k} $, then
\begin{equation}\label{herm:errorineq}
    \eta^{{\rm Herm}}(G,\lambda) =  \frac{1}{\sqrt{\widehat \lambda_{\max}}}.
\end{equation}
\end{theorem}

\begin{proof}
From Corollary~\ref{cor1} when $\mathbb S = \left({\rm Herm}(n)\right)^{d+k+1}$, the problem of computing $\eta^{{\rm Herm}}(G,\lambda)$ can be simplified using minimal norm Hermitian mappings. In view of~\cite[Theorem 2.2.3]{Adh08}, for any $v_{A_0},\ldots, v_{A_d},v_{E_1},\ldots,v_{E_k} \in \mathbb{C}^{n}$ with $v_{\lambda}:= \sum_{p=0}^d \lambda^p v_{A_p} + \sum_{j=1}^k w_j(\lambda) v_{E_j} \neq 0$ there exist Hermitian matrices $\Delta A_p$ and $\Delta E_j$ for $p=0,\ldots,d$ and $j=1,\ldots,k$ such that 
\begin{equation}\label{map:herm}
\Delta A_p Mv_{\lambda}=v_{A_p},\;\; \Delta E_jMv_{\lambda}=v_{E_j}
\end{equation}
if and only if $\imag{v_{A_p}^*Mv_{\lambda}}=0$ and $\imag{v_{E_j}^*Mv_{\lambda}}=0$. The latter $d+k+1$ conditions can be collectively written in terms of $d+k+1$ Hermitian constraints $u^*H_ju=0$ for $j=0,\ldots,d+k$, where 
\begin{equation}\label{matrix:herm}
    u := {\big[
v_{A_0}^T ~ \cdots ~ v_{A_d}^T ~ v_{E_1}^T ~ \cdots ~ v_{E_k}\big]}^T \quad \text{and}\quad H_j=i(\Lambda^*e_{j+1}^*\otimes M^*-e_{j+1}\Lambda \otimes M).
\end{equation}
 The first $d+1$ constraints $u^*H_pu=0$ for $p=1,\ldots,d$ in~\eqref{matrix:herm} are corresponding to the conditions $\imag{v_{A_p}^*Mv_{\lambda}}=0$ and the last $k$ constraints $u^*H_pu=0$ for $p=d+1,\ldots,d+k$ in~\eqref{matrix:herm} are corresponding to the conditions $\imag{v_{E_p}^*Mv_{\lambda}}=0$. In view of~\cite[Theorem 2.2.3]{Adh08}, the minimal norm $\nrm{(\Delta A_0,\ldots,\Delta A_d,\Delta E_1,\ldots,\Delta E_k)}$ for a fixed tuple $(\Delta A_0,\ldots,\Delta A_d,\Delta E_1,\ldots,\Delta E_k)$ is then given by 
\begin{equation}\label{norm:herm}
    {\nrm{(\Delta A_0,\ldots,\Delta A_d,\Delta E_1,\ldots,\Delta E_k)}}^2  =\sum_{p=0}^d\frac{\|v_{A_i}\|^2}{\|Mv_{\lambda}\|^2}+\sum_{j=1}^k\frac{\|v_{E_j}\|^2}{\|Mv_{\lambda}\|^2} = \frac{u^*u}{u^*Ju},
\end{equation}
where $J = \Lambda^* \Lambda \otimes M^*M$. Also note that 
$ u^*Ju={\|Mv_\lambda\|}^2 \neq 0$ if and only if $\Leftrightarrow v_\lambda \neq 0$, since $M$ is invertible. Using the above observations in Corollary~\ref{cor1} yields that 
\begin{eqnarray}\label{herm:error2}
    (\eta^{{\rm Herm}}(G,\lambda))^2&=&\left(\sup\left\{\frac{u^*Ju}{u^*u} \; | \;u \in \C^{n(d+k+1)},\,u^*Ju \neq 0,\,  u^*H_pu=0,~p=0,\ldots,d+k \right\}\right)^{-1} \nonumber \\
  &=&\left(\sup\left\{\frac{u^*Ju}{u^*u} \; | \; u \in \C^{n(d+k+1)}\setminus \{0\},~u^*H_pu=0, ~ p=0,\ldots,d+k \right\}\right)^{-1},
\end{eqnarray}
where the last equality follows due to the fact that $\eta^{{\rm Herm}}(G,\lambda)$ is finite and therefore the supremum in the right hand side of~\eqref{herm:error2} will not be attained at $u$ for which $u^*Ju =0$. Thus including the vectors $u$ such that $u^*Ju=0$ will not change the value of the supremum in~\eqref{herm:error2}.

Now the result follows directly from Theorem~\ref{theorem:optimize2} if we show that any nonzero linear combination of $\alpha_0 H_0+\ldots+\alpha_{d+k}H_{d+k}$, $(\alpha_0,\ldots,\alpha_{d+k}) \in \R^{d+k+1}$ is indefinite. If possible suppose that there exists $\alpha:={\big [\alpha_0, \ldots, \alpha_{d+k}\big]}^T \in \mathbb{R}^{d+k+1} \setminus \{0\}$ such that 
\begin{equation}
   H:=\sum_{j=0}^{d+k} \alpha_jH_j  =\sum_{j=0}^{d+k} i \alpha_j(\Lambda^*e_{i+1}^*\otimes M^* - e_{i+1}\Lambda \otimes M)
         = i(\Lambda^*\alpha^T \otimes M^* - \alpha \Lambda \otimes M)
\end{equation}
is semidefinite. Define the bidiagonal matrix
\begin{equation}
    Q:=\begin{bmatrix} 
1 & -\lambda &  &  &  &  &  &\\
 0& \ddots & \ddots &  &  &  &  &  \\
 &  &  &  -\lambda &  &  &  &  \\
 &  &  & 1 & -\frac{w_1(\lambda)}{\lambda^d} &  &  &  \\
 &  &  &  & 1 & -\frac{w_2(\lambda)}{w_1(\lambda)} &  &  \\
 &  &  &  &  & \ddots & \ddots & \\
 &  &  &  &  &  &  & -\frac{w_k(\lambda)}{w_{k-1}(\lambda)} \\
 &  &  &  &  &  &  &   1 \\
\end{bmatrix}
\quad \text{and}\quad a:=\mat{c}a_0\\a_1\\ \vdots\\a_{d+k} \rix=Q^*\alpha.
\end{equation}
Then we have $\Lambda Q = e_1^T$ and 
\begin{equation}
    (Q\otimes I_n)^*H(Q \otimes I_n)  = i(e_1a^* \otimes M^* - a e_1^T \otimes M)= i\begin{bmatrix}
    a_0M-\overline{a_0}M^* & -\overline{a_1}M^* & \cdots & -\overline{a_{d+k}M^*}\\
    a_1M &0 &\cdots &0 \\
    \vdots &\vdots & &\vdots \\
    a_{d+k}M &0 &\cdots &0 \\
    \end{bmatrix}.
\end{equation}
Since $H$ is semidefinite and $M$ is invertible it follows that $a_1= \cdots =a_{d+k}=0$, i.e., $a=a_0e_1$. In particular $a_1=0$ implies that $\alpha_1-\overline{\lambda}\alpha_0 =0 $ which further implies that $\alpha_1=\alpha_0=0$ for $\lambda \in \mathbb{C}\setminus \mathbb{R}$ and hence $a_0=\alpha_0=0$. This implies that  $a=0$ and thus $\alpha=0$ which is a contradiction. Thus result follows by applying Theorem~\ref{theorem:optimize2} in~\eqref{herm:error2}.
\end{proof}

As done in~\cite{MR3194659} for Hermitian-related matrix polynomials, we can also derive the structured backward errors for eigenvalues of RMPs with skew-Hermitian, $*$-even, and $*$-odd structures using Theorem~\ref{thm:Herm} of the Hermitian case. This follows by exploiting the fact that 
$G(z)$ is a skew-Hermitian RMP if and only if $R(z)=iG(z)$ is Hermitian. Similarly, it is easy to observe from Table~\ref{tab:my_label} that $G(z)$ is $*$-even or $*$-odd if and only if $R(z)=G(iz)$ or $R(z)=iG(iz)$ is Hermitian. To derive formulas for the structured eigenvalue backward error for $*$-even RMPs, we define
\begin{eqnarray*}
		\mathbb{S}_{e}: &=\Big\{\left(\Delta A_{0}, \ldots,\Delta A_{d},\Delta E_{0},\ldots, \Delta E_{k}\right) \in (\C^{n,n})^{d+k+1} :~ \Delta A_{p}^*=(-1)^{p}\Delta A_p ,\,\\
	&	p=0,\ldots, d,~\Delta E_j \in {\rm Herm}(n),~ j=1,\ldots, k\Big\},
\end{eqnarray*}
when $(w_j(-z))^*=w_j(\overline z)$, and 
\begin{eqnarray*}
		\mathbb{S}_{e}: &=\Big\{\left(\Delta A_{0}, \ldots,\Delta A_{d},\Delta E_{0},\ldots, \Delta E_{k}\right)  \in (\C^{n,n})^{d+k+1}:~  \Delta A_{p}^*=(-1)^{p}\Delta A_p,\\
	&	p=0,\ldots, d,~\Delta E_j \in {\rm SHerm}(n),~ j=1,\ldots, k\Big\},
\end{eqnarray*}
when $(w_j(-z))^*=-w_j(\overline z)$. Similary, for the $*$-odd structure we define
\begin{eqnarray*}
		\mathbb{S}_{o}: &=\Big\{\left(\Delta A_{0}, \ldots,\Delta A_{d},\Delta E_{0},\ldots, \Delta E_{k}\right)  \in (\C^{n,n})^{d+k+1}:~ \Delta A_{p}^*=(-1)^{p+1}\Delta A_p,\\
	&	p=0,\ldots, d,~\Delta E_j \in {\rm Herm}(n),~ j=1,\ldots, k\Big\},
\end{eqnarray*}
when $(w_j(-z))^*=-w_j(\overline z)$, and 
\begin{eqnarray*}
		\mathbb{S}_{o}: &=\Big\{\left(\Delta A_{0}, \ldots,\Delta A_{d},\Delta E_{0},\ldots, \Delta E_{k}\right)  \in (\C^{n,n})^{d+k+1}:~ \Delta A_{p}^*=(-1)^{p+1}\Delta A_p,\\
	&	p=0,\ldots, d,~\Delta E_j \in {\rm SHerm}(n),~ j=1,\ldots, k\Big\},
\end{eqnarray*}
when $(w_j(-z))^*=w_j(\overline z)$. Then from~\eqref{def:error} the corresponding backward errors are denoted by $\eta^{{\rm SHerm}}(G,\lambda)$ when $\mathbb S= ({\rm SHerm})^{d+k+1}$, $\eta^{{\rm even}_*}(G,\lambda)$ when $\mathbb S= {\mathbb S}_e$, and $\eta^{{\rm odd}_*}(G,\lambda)$ when $\mathbb S= {\mathbb S}_o$. We have the following result.

\begin{theorem}
Let $\mathbb S\in \left\{ {\mathbb S}_e,{\mathbb S}_o,({\rm SHerm})^{d+k+1} \right\}$ and let 
$G(z)$ be a RMP of the form~\eqref{rmatrix}, where the coefficient matrices $(A_0,\ldots,A_d,E_1,\ldots,E_k) \in \mathbb S$. Let $\lambda \in \C$ satisfying~\eqref{assump} and define $R(z)=G(iz)$. Then the following holds.
\begin{itemize}
\item when $\mathbb S= ({\rm SHerm})^{d+k+1}$, we have
\[
\eta^{{\rm SHerm}}(G,\lambda)=\eta^{{\rm Herm}}(iG,\lambda).
\]
\item when $\mathbb S= {\mathbb S}_e$, we have
\[
\eta^{{\rm even}_*}(G,\lambda)=\eta^{{\rm Herm}}(R,\frac{\lambda}{i}).
\]
\item when $\mathbb S= {\mathbb S}_o$, we have
\[
\eta^{{\rm odd}_*}(G,\lambda)=\eta^{{\rm Herm}}(iR,\frac{\lambda}{i}).
\]
\end{itemize}
\end{theorem}

\begin{remark}{\rm \label{rem:optpertherm}
 For the construction of the optimal perturbation $\Delta G(z)=\sum_{p=0}^d z^p {\Delta A}_p + \sum_{j=1}^kw_j(z){\Delta E}_j$ with 
$\nrm{({\Delta A}_0,\ldots,{\Delta A}_d,{\Delta E}_1,\ldots,{\Delta E}_k)}=
\eta^{{\rm Herm}}(G,\lambda)$ and $\text{det}(G(\lambda)-\Delta G(\lambda))=0$, when $\widehat \lambda _{\max}$ is a simple eigenvalue
of $J+\sum_{j=0}^{d+k}\widehat t_j H_j$ in Theorem~\ref{thm:Herm}, we first compute the corresponding eigenvector $u \in {(\C^n)}^{d+k+1}$ such that $u^*H_ju=0$ for all $j=0,\ldots,d+k$. By writing $u=\big[v_{A_0}^T~\cdots~v_{A_d}^T~v_{E_1}^T~\cdots~v_{E_k}^T\big]^T$ with $v_{A_p}^T,v_{E_j}^T \in \C^n$ for $p=0,\ldots,d$ and $j=1,\ldots,k$, and $v_\lambda=\sum_{p=0}^d \lambda^p v_{A_p}+\sum_{j=1}^k w_j(\lambda) v_{E_j}$, the perturbation matrices ${\Delta E}_p$'s and ${\Delta A}_j$'s may be obtained from~\cite[Theorem 2.2.3]{Adh08} such that 
\[
{\Delta A}_pMv_{\lambda}=v_{A_p} ~\text{for}~p=0,\ldots,d,\quad \text{and}\quad {\Delta 
E}_jMv_{\lambda}=v_{E_j}~\text{for}~j=1,\ldots,k.
\]
A similar remark also holds for other Hermitian related structures. 
}
\end{remark}

\section{Palindromic structure } \label{sec:pal}

In this  section, we consider RMPs $G(z)$ of the form~\eqref{rmatrix} with  a palindromic structure, see Table~\ref{tab:my_label}. We use the term
$\bullet$-palindromic, where $\bullet \in \{*,T\}$ to make the statements that are valid for both $*$-palindromic and T-palindromic RMPs. Thus $G(z)$ is $\bullet$-palindromic RMP if the coefficient matrices in $G(z)$ satisfy that $A_p^\bullet = A_{d-p}$ for $j=0,\ldots,d$, and $E_j^\bullet =E_j$ and $(w_j(z))^{\bullet}=(z^{\bullet})^dw_j(\frac{1}{z^{\bullet}})$ for $j=1,\ldots,k$.  To define formulas for the structured eigenvalue backward errors for $\bullet$-palindromic RMPs, we define
\begin{eqnarray}
{\rm Pal}_\bullet := &\Big\{ (\Delta A_0,\ldots,\Delta A_d,\Delta E_1,\ldots,\Delta E_k) \in (\C^{n,n})^{d+k+1}:~(\Delta A_p)^\bullet=\Delta A_{d-p}, \,\\ 
& p=0,\ldots,d, \, E_j^\bullet =E_j,\, j=1,\ldots,k
\Big\}.
\end{eqnarray}
Then the corresponding $\bullet$-palindromic eigenvalue backward error is defined by~\eqref{def:error} when $\mathbb S = {\rm Pal}_\bullet$ and is denoted by $\eta^{{\rm pal}_\bullet}(G,\lambda)$.

We note that when $\lambda \in \C \setminus \{0\}$ such that $|\lambda|=1$ ($\lambda=\pm 1$), then there is no difference between structured and unstructured eigenvalue backward errors of $*$-palindromic  (T-palindromic) RMPs. This fact was shown in~\cite{Ahm19}. However, the situation is completely different if $|\lambda| \neq 1$ (when $\bullet=*$) and $\lambda\neq \pm 1$ (when $\bullet=T$).

Consider a $\bullet$-palindromic RMP $G(z)$ of the form~\eqref{rmatrix}, $\lambda \in \C$ be such that $M:=(G(\lambda))^{-1}$ exists, and let $m:=\lfloor \frac{d-1}{2} \rfloor$. Then from Corollary~\ref{cor1}, we have that 
 \begin{eqnarray}\label{eq:palcorodd}
& \eta^{{\rm pal}_\bullet}(G,\lambda):=\inf\bigg\{  \nrm{(\Delta A_{0}, \ldots, \Delta A_{d},\Delta E_1, \ldots, \Delta E_k)} \; 
 :\;(\Delta A_{0}, \ldots, \Delta A_{d},\Delta E_1, \ldots, \Delta E_k)\in {\rm Pal}_\bullet, \nonumber\\
            &\exists\, v_{A_0}, \ldots,v_{A_d}, v_{E_1},\ldots,v_{E_k} \in \C^n,
            \; v_{\lambda}:=\sum_{p=0}^d\lambda^iv_{A_p}+\sum_{j=1}^kw_j(\lambda)v_{E_j}\neq 0,\,\Delta A_p Mv_{\lambda}=v_{A_p},\nonumber\\
            & \,(\Delta A_p)^\bullet Mv_{\lambda}=v_{A_{d-p}}\,p=0,\ldots,m,\;\text{and}~\, v_{E_j}=\Delta E_jMv_{\lambda},\,j=1,\ldots,k\bigg\},
 \end{eqnarray}
when $d$ is odd, and 
 \begin{eqnarray}\label{eq:palcoreven}
& \eta^{{\rm pal}_\bullet}(G,\lambda):=\inf\bigg\{  \nrm{(\Delta A_{0}, \ldots, \Delta A_{d},\Delta E_1, \ldots, \Delta E_k)} \; 
 :\;(\Delta A_{0}, \ldots, \Delta A_{d},\Delta E_1, \ldots, \Delta E_k)\in {\rm Pal}_\bullet, \nonumber\\
            &\exists\, v_{A_0}, \ldots,v_{A_d}, v_{E_1},\ldots,v_{E_k} \in \C^n ~s.t
            \; v_{\lambda}:=\sum_{p=0}^d\lambda^iv_{A_p}+\sum_{j=1}^kw_j(\lambda)v_{E_j}\neq 0,\,\Delta A_p Mv_{\lambda}=v_{A_p},\nonumber\\
            & \,(\Delta A_p)^\bullet Mv_{\lambda}=v_{A_{d-p}}\,p=0,\ldots,m,\;\Delta A_{\frac{d}{2}}Mv_{\lambda}=v_{A_{\frac{d}{2}}},\,\text{and}~\, v_{E_j}=\Delta E_jMv_{\lambda},\,j=1,\ldots,k\bigg\},
 \end{eqnarray}
when $d$ is even.

In view of~\cite[Theorem 2.1]{MehMS17}, for any $v_{A_0},\ldots, v_{A_d},v_{E_1},\ldots,v_{E_k} \in \mathbb{C}^{n}$ with $v_{\lambda}:= \sum_{p=0}^d \lambda^p v_{A_p} + \sum_{j=1}^k w_j(\lambda) v_{E_j} \neq 0$ there exist $\Delta A_p$ for $p=0,\ldots,m$ such that 
\begin{equation}\label{map:herm}
\Delta A_p Mv_{\lambda}=v_{A_p} \quad \text{and}\quad  (\Delta A_p)^\bullet Mv_{\lambda}=v_{A_{d-p}}
\end{equation}
if and only if $v_{A_p}^{\bullet}Mv_{\lambda}=(Mv_{\lambda})^{\bullet}v_{A_{d-p}}$, and the minimal norm of such a $\Delta A_p$ is given by 
${\|\Delta A_p\|}=\max\left\{ \frac{\|v_{A_p}\|}{\|Mv_{\lambda}\|},\frac{\|v_{A_{d-p}}\|}{\|Mv_{\lambda}\|} \right\}$.

Similarly, when $\bullet=*$, $\Delta E_j$ is Hermitian for all $j=1,\ldots,k$ and from~\cite[Theorem 2.2.3]{Adh08}, there exists $\Delta E_j$ for $j=1,\ldots,k$ such that $\Delta E_jMv_{\lambda}=v_{E_j}$ if and only if $\imag{v_{E_j}^*Mv_{\lambda}}=0$,  or equivalently, $v_{E_j}^*Mv_{\lambda}=(Mv_{\lambda})^*V_{E_j}$. The minimal norm of such a $\Delta E_j$ is given by 
$\|\Delta E_j\|=\frac{\|v_{E_j}\|}{\|Mv_{\lambda}\|}$. Note that when $\bullet=T$, then there always exists symmetric matrices $\Delta E_j$ for $j=1,\ldots,k$ such that $\Delta E_jMv_{\lambda}=v_{E_j}$ and the minimal norm of such a $\Delta E_j$ satisfies that
$\|\Delta E_j\|=\frac{\|v_{E_j}\|}{\|Mv_{\lambda}\|}$. Also when $\bullet =*$ and $d$ is even, then $\Delta A_{\frac{d}{2}}$ is Hermitian. There exists $\Delta A_{\frac{d}{2}}$ such that  $\Delta A_{\frac{d}{2}}Mv_{\lambda}=v_{A_{\frac{d}{2}}}$ iff $\imag{v_{A_{\frac{d}{2}}}^*Mv_{\lambda}}=0$,  or equivalently, $v_{A_{\frac{d}{2}}}^*Mv_{\lambda}=(Mv_{\lambda})^*V_{A_{\frac{d}{2}}}$. 
When $d$ is even and $\bullet = T$, then there always exists a symmetric matrix $\Delta A_{\frac{d}{2}}$ such that $\Delta A_{\frac{d}{2}}Mv_{\lambda}=v_{\frac{d}{2}}$.  The minimal norm of such a $\Delta A_{\frac{d}{2}}$ satisfies that 
$\Delta A_{\frac{d}{2}}=\frac{\big\|v_{A_{\frac{d}{2}}}\big\|}{\|Mv_{\lambda}\|}$.

By using the above mapping conditions in~\eqref{eq:palcorodd}
when $d$ is odd and in~\eqref{eq:palcoreven} when $d$ is even, the 
minimal norm $\nrm{(\Delta A_0,\ldots,\Delta A_d,\Delta E_1,\ldots,\Delta E_k)}$ for a fixed tuple $(v_{A_0},\ldots, v_{A_d},v_{E_1},\ldots,v_{E_k} )$ is then given by 
\begin{equation}\label{eq:funch}
    \nrm{(\Delta A_0,\ldots,\Delta A_d,\Delta E_1,\ldots,\Delta E_k)}^2=h(v_{A_0},\ldots,v_{A_d},v_{E_1},\ldots,v_{E_k}),
\end{equation}
where 
\begin{equation}
    h(v_{A_0},\ldots,v_{E_k})=\begin{cases}
        \sum_{p=0}^m 2 \max\left\{\frac{\|v_{A_p}\|^2}{\|Mv_{\lambda}\|^2},\frac{\|v_{A_{d-p}}\|^2}{\|Mv_{\lambda}\|^2}\right\}+\sum_{j=1}^k\frac{\|v_{E_j}\|^2}{\|Mv_{\lambda}\|^2} & \quad \text{if d odd},\\
        \sum_{p=0}^m 2 \max\left\{\frac{\|v_{A_p}\|^2}{\|Mv_{\lambda}\|^2},\frac{\|v_{A_{d-p}}\|^2}{\|Mv_{\lambda}\|^2}\right\}+\frac{\|v_{\frac{d}{2}}\|^2}{\|Mv_{\lambda}\|^2}+\sum_{j=1}^k\frac{\|v_{E_j}\|^2}{\|Mv_{\lambda}\|^2} & \quad \text{if d even}.
    \end{cases}
\end{equation}
Let us define, 
\begin{eqnarray}\label{constraint:Keven}
    &\mathcal{K}:=\big\{(v_{A_0},\ldots,v_{E_k}) \in (\C^{n})^{d+k+1}\; : ~ v_{\lambda}\neq 0,\;  v_{A_p}^{\bullet}Mv_{\lambda} =(Mv_{\lambda})^{\bullet}v_{A_{d-p}}, \nonumber\\
    &\; p=0,\ldots,m,~v_{E_j}^{\bullet}Mv_{\lambda}=(Mv_{\lambda})^{\bullet}v_{E_j},\; j=1,\ldots, k \big\},
\end{eqnarray}
when $\bullet = T$, or if d is odd and $\bullet = *$, and
\begin{eqnarray}\label{constraint:Kodd}
        \mathcal{K}:=&\big\{(v_{A_0},\ldots,v_{E_k}) \in (\C^n)^{d+k+1} \; :\; v_{\lambda}\neq 0,~v_{A_p}^{\bullet}Mv_{\lambda}=(Mv_{\lambda})^{\bullet}v_{A_{d-p}},~p=0,\ldots,m, \nonumber \\
    & v_{A_{\frac{d}{2}}}^{\bullet}Mv_{\lambda}  =(Mv_{\lambda})^{\bullet}v_{A_{\frac{d}{2}}}, ~ v_{E_j}^{\bullet}Mv_{\lambda}=(Mv_{\lambda})^{\bullet}v_{E_j},~ j=1,\ldots,k \big\}
\end{eqnarray}
otherwise, (i.e., when $\bullet = *$ and d is even). Then by using~\eqref{eq:funch}-\eqref{constraint:Kodd} in~\eqref{eq:palcorodd} or~\eqref{eq:palcoreven}, we obtain
\begin{equation}\label{eq:partreform}
    \eta^{{\rm pal}_\bullet}(G,\lambda)^2=\inf\left\{ h(v_{A_0},\ldots,v_{E_k})\; : \; (v_{A_0},\ldots,v_{E_k})\in \mathcal{K} \right\}.
\end{equation}
Our idea is to use the strategy suggested in~\cite{MR3194659} and reformulate the optimization problem in~\eqref{eq:partreform} equivalently in terms of minimizing the Rayleigh quotient of some Hermitian matrix with respect to certain constraints involving Hermitian or symmetric matrices. This can be achieved using the following lemma, proof of which is analogous to~\cite[Lemma 3.1]{MR3335496} and therefore skipped. 
\begin{lemma}\label{pal:normchange}
    Let $G(z)$ be a $\bullet$-palindromic RMP of the form~\eqref{rmatrix} and $\lambda \in \mathbb{C} \setminus\{0\}$ satisfying~\eqref{assump}. Suppose that $M:=(G(\lambda))^{-1}$ exists and $m:=\left\lfloor\frac{d-1}{2}\right\rfloor $. Then
    \begin{equation}
        (\eta^{\mathrm{pal\bullet}}(G, \lambda))^{2}=\inf \left\{g(v_{A_0}, \ldots, v_{E_k}) ~:~(v_{A_0}, \ldots, v_{E_k}) \in \mathcal{K}\right\},
    \end{equation}
    where
    \begin{equation}
        g(v_{A_0}, \ldots, v_{E_k}):= 
        \begin{cases}
            \sum_{j=0}^{m} \frac{2\left(\|v_{j}\|^{2}+|\lambda|^{d-2 j}\|v_{d-j}\|^{2}\right)}{\left(1+|\lambda|^{d-2 j}\right)\|M v_{\lambda}\|^{2}}+\sum_{i=1}^k\frac{\|v_{E_i}\|^2}{\|Mv_{\lambda}\|^2} & \text { if } d \text { is odd } \\
            \sum_{j=0}^{m} \frac{2\left(\|v_{j}\|^{2}+|\lambda|^{d-2 j}\|v_{d-j}\|^{2}\right)}{\left(1+|\lambda|^{d-2 j}\right)\|M v_{\lambda}\|^{2}}+\frac{\|v_{\frac{d}{2}}\|^{2}}{\|M v_{\lambda}\|^{2}}+\sum_{i=1}^k\frac{\|v_{E_i}\|^2}{\|Mv_{\lambda}\|^2} \quad & \text { if } d \text { is even }
        \end{cases}
    \end{equation}
    and $\mathcal{K}$ is as defined in~\eqref{constraint:Keven} and~\eqref{constraint:Kodd}.
\end{lemma}

By following the steps of Section~\ref{sec:herm} for Hermitian structure, the function $g(v_{A_0}, \ldots, v_{E_k})$ in Lemma~\ref{pal:normchange} can be written as a Rayleigh quotient of some Hermitian matrix and the conditions in the set $\mathcal K$ can also be simplified as some Hermitian or symmetric constraints depending on whether $\bullet=*$ or $\bullet =T$.

For $m=\left\lfloor\frac{d-1}{2}\right\rfloor$, define $\gamma_{j1}:=\sqrt{\frac{2}{1+|\lambda|^{d-2 j}}}, \gamma_{j2}:=\sqrt{\frac{2|\lambda|^{d-2j}}{1+|\lambda|^{d-2j}}}$,  $j=0, \ldots, m$,

\begin{equation}\label{Gamma}
    \Gamma:= 
    \begin{cases}
        \diag\left(\gamma_{01}, \ldots, \gamma_{m 1}, \gamma_{m 2}, \ldots, \gamma_{02},1,\ldots,1\right) \otimes I_{n} & \text { if } d \text { is odd } \\
        \diag\left(\gamma_{01}, \ldots, \gamma_{m 1}, 1, \gamma_{m 2}, \ldots, \gamma_{02},1,\ldots,1\right) \otimes I_{n} & \text { if } d \text { is even }
    \end{cases},
\end{equation}
and $\Lambda:=\left[1, \lambda, \ldots, \lambda^{m},w_1(\lambda),\ldots,w_k{\lambda}\right] \in \mathbb{C}^{1 \times(d+k+1)}$. Then we have
\begin{equation}\label{matrix:palv}
    g\left(v_{A_0}, \ldots, v_{E_k}\right)=\frac{v^{*} \Gamma^{2} v}{v^{*} \widetilde{J} v},
\end{equation}
where  $ \widetilde{J}:=\left(\Lambda^{*} \Lambda\right) \otimes\left(M^{*} M\right)$ and $ v=\left[v_{A_0}^{T}, \ldots,v_{A_d}^{T}, v_{E_1}^T,\ldots,v_{E_k}^{T}\right]^{T}$.  Further 
$v^{*} \widetilde{J} v=\left\|M v_{\lambda}\right\|^{2} \neq 0$ if and only if  $v_{\lambda} \neq 0 $, since $M$ is invertible.

Next we simplify the constraints in the set $\mathcal{K}$. Note that the constraint $v_{A_p}^{\bullet}Mv_{\lambda}=(Mv_{\lambda})^{\bullet}v_{A_{d-p}}$ for $p=0,\ldots,m$ can be equivalently written as 
\begin{equation}\label{matrix:palCi}
v^*\widetilde{C}_pv=0\quad  \text{and}\quad \widetilde{C}_p = \Lambda^{\bullet}e_{p+1}^{\bullet}\otimes M^{\bullet}-e_{d-p+1}\Lambda \otimes M.
\end{equation}
Similarly, for $j=1,\ldots,k$ the constraint  $v_{E_j}^{\bullet}Mv_{\lambda}=(Mv_{\lambda})^{\bullet}v_{E_{j}}$  is equivalent to  
\begin{equation}\label{matrix:palEH}
v^*\widetilde{H}_{d+j}v=0~\text{and}\quad \widetilde{H}_{d+j} = i(\Lambda^*e_{d+j+1}^*\otimes M^*-e_{d+j+1}\Lambda \otimes M).
\end{equation}
Note that the condition $v_{E_j}^{\bullet}Mv_{\lambda}=(Mv_{\lambda})^{\bullet}v_{E_{j}}$ trivially holds if $\bullet = T$. 
When $d$ is even, the constraint  $v_{A_{\frac{d}{2}}}^*Mv_{\lambda} =(Mv_{\lambda})^*v_{A_{\frac{d}{2}}}$ reduces to 
\begin{equation}\label{matrix:palCd2}
v^*\widetilde{C}_{\frac{d}{2}}v=0 \quad \text{and}\quad \widetilde{C}_{\frac{d}{2}}=i(\Lambda^*e_{\frac{d}{2}+1}^*\otimes M^*-e_{\frac{d}{2}+1}\Lambda \otimes M).
\end{equation}
In view of~\eqref{matrix:palCi}-\eqref{matrix:palCd2}, the set $\mathcal K$ becomes 
\begin{equation}
\mathcal{K}=\left\{ (v_{A_0},\ldots, v_{E_k})\in (\C^n)^{d+k+1} :~ v_{\lambda}\neq 0, \; v^T\widetilde{C}_pv=0, \; p=0,\ldots, m \right\}
\end{equation}
when $\bullet = T$, 
and
\begin{equation}
\mathcal{K}=\begin{cases}\left\{ (v_{A_0},\ldots, v_{E_k})\in (\C^n)^{d+k+1}  :~v_{\lambda}\neq 0\; , \; v^*\widetilde{C}_pv=0,\; p=0,\ldots,m+1,\; v^*\widetilde{H}_{d+j}v=0,~ j=1,\ldots,k \right\}\\ \hspace{11cm} \text{if d is even}\\
\left\{ (v_{A_0},\ldots, v_{E_k})\in (\C^n)^{d+k+1}  :~v_{\lambda}\neq 0, \; v^*\widetilde{C}_pv=0,~p=0,\ldots,m, v^*\widetilde{H}_{d+j}v=0,~ j=1,\ldots,k \right\} \\ \hspace{11cm} \text{if d is odd}
\end{cases}
\end{equation}
when $\bullet = *$.
To make the constraints $\widetilde C_p$, for $p=0,\ldots,m$ in $\mathcal K$ Hermitian (when $\bullet =*$) and symmetric (when $\bullet=T$), we set  $u:=\Gamma v$ and define for $p=0,\ldots,m$, 
\begin{eqnarray}\label{matrix:palHS}
    H_{p}:=\Gamma^{-1} (\widetilde{C}_{p}+\widetilde{C}_p^*) \Gamma^{-1},\quad  H_{d-p}:=i\Gamma^{-1} (\widetilde{C}_{p}-\widetilde{C}_p^*) \Gamma^{-1}, \quad 
    S_p=\Gamma^{-1}(\widetilde{C}_p+\widetilde{C}_p^T)\Gamma^{-1},
\end{eqnarray}
and set 
\begin{equation}\label{matrix:palJ}
    J:=\Gamma^{-1} \widetilde{J} \Gamma^{-1}, \quad H_{\frac{d}{2}}:=\Gamma^{-1} \widetilde{C}_{\frac{d}{2}} \Gamma^{-1}, \quad \text{and}\quad
    H_{d+j}:=\Gamma^{-1}\widetilde H_{d+j} \Gamma^{-1},~\text{for}~ j=1,\ldots,k.
\end{equation}
Then it is easy to observe that for $p=0,\ldots,m$
\begin{eqnarray}\label{matrix:C-H}
v^{*}\widetilde{C}_pv=0 & \iff & v^{*}(\widetilde{C}_p+\widetilde{C}_p^{*})v=0 \quad \text{and} \quad  v^{*}(\widetilde{C}_p-\widetilde{C}_p^{*})v=0 \nonumber \\
        & \iff &u^{*}H_pu=0 \quad \text{and} \quad  u^{*}H_{d-p}u=0,
\end{eqnarray}
and 
\begin{equation}\label{matrix:C-S}
v^{T}\widetilde{C}_pv=0 \iff  v^{T}(\widetilde{C}_p+\widetilde{C}_p^{T})v=0  \iff u^{T}S_pu=0.
\end{equation}
Thus applying~\eqref{matrix:palv} and~\eqref{matrix:C-H}-\eqref{matrix:C-S} in Lemma~\ref{pal:normchange} yields that 
\begin{equation}\label{pal*:error1}
    (\eta^{{\rm pal}_*}(G,\lambda))^2=\left(\sup\left\{ \frac{u^*Ju}{u^*u} :~ u \in \C^{n(d+k+1)}\setminus\{0\},~ u^*H_ju=0,~ j=0,\ldots, d+k  \right\}\right)^{-1},
\end{equation}
when $\bullet = *$, and 
\begin{equation}\label{palT:error1}
    (\eta^{{\rm pal}_T}(G,\lambda))^2=\left(\sup\left\{ \frac{u^*Ju}{u^*u} :~ u \in \C^{n(d+k+1)}\setminus\{0\},~ u^TS_ju=0,~ j=0,\ldots, m \right\}\right)^{-1},
\end{equation}
when $\bullet = T$. We note that in~\eqref{pal*:error1} and~\eqref{palT:error1}, the condition  $v_{\lambda}\neq0$, or, equivalently $u^*{J}u\neq 0$ was dropped because the backward error 
$\eta^{{\rm pal}_\bullet}(G,\lambda)$ is finite, and thus the optimal in~\eqref{pal*:error1} and~\eqref{palT:error1} will not be attained at the vectors $u$ for which $u^*{J}u= 0$.  Thus allowing the vectors $u$ with  $u^*{J}u= 0$ would not change anything. 

On applying Theorem~\ref{theorem:optimize2} to~\eqref{pal*:error1}, we  derive a computable formula for the $*$-palindromic backward error $\eta^{{\rm pal}_*}(G, \lambda)$ in the following result,  

\begin{theorem}\label{thm:mainpalin}
    Let $G(z)$ be a $*$-palindromic RMP of the form~\eqref{rmatrix} and $\lambda \in \mathbb{C} \backslash\{0\}$ satisfying~\eqref{assump} and $|\lambda| \neq 1$. Suppose that $M:=(G(\lambda))^{-1} $ exists. Then for $J$ and $H_{j}$ for $j=0, \ldots, d+k$, as defined in~ \eqref{matrix:palHS} and~\eqref{matrix:palJ}, we have that 
\begin{equation*}
\widehat \lambda_{\max }:=\min _{t_{0}, \ldots, t_{d+k} \in \mathbb{R}} \lambda_{\max }\left(J+t_{0} H_{0}+\cdots+t_{d+k} H_{d+k}\right)
\end{equation*}
    is attained for some $\left(\hat{t}_{0}, \ldots, \hat{t}_{d+k}\right) \in \mathbb{R}^{d+k+1} $, and $\eta^{{\rm pal}_*}(G, \lambda)\geq\frac{1}{\sqrt{\widehat \lambda_{\max }}}$. 
Further,  if $\widehat \lambda_{\max }$ is a simple eigenvalue of $J+\hat t_{0} H_{0}+\cdots+\hat t_{d+k} H_{d+k}$, then 
\begin{equation}\label{eq:palerrorineq}
        \eta^{{\rm pal}_*}(G, \lambda)=\frac{1}{\sqrt{\widehat \lambda_{\max }}}.
\end{equation}
\end{theorem}
\proof
In view of~\eqref{pal*:error1}, the proof follows from Theorem 
~\ref{theorem:optimize2} if we show that any nonzero linear combination 
$\alpha_0 H_0 +\cdots+\alpha_{d+k}H_{d+k}$, $(\alpha_0,\ldots,\alpha_{d+k}) \in \R^{d+k+1} \setminus \{0\}$ is indefinite, or equivalently, $\Gamma\left(\alpha_0 H_0 +\cdots+\alpha_{d+k}H_{d+k}\right)\Gamma$ is indefinite, where $\Gamma$ is defined in~\eqref{Gamma}. On the contrary, suppose that 
$H:=\sum_{j=0}^{d+k} \hat \alpha_j \Gamma H_j\Gamma $ is semidefinite for some 
$(\hat \alpha_0,\ldots,\hat \alpha_{d+k}) \in \R^{d+k+1}$. Then we can write $H$ as
\begin{equation}
        H = i(\Lambda^*{\widehat \alpha}^T \otimes M^* - \widehat \alpha \Lambda \otimes M),
\end{equation}
where 
\begin{equation}
    \widehat \alpha :=\begin{cases}
        \big[\hat \alpha_0-i \hat \alpha_d, \ldots, \hat\alpha_m-i\hat\alpha_{d-m}, -(\hat\alpha_m+i\hat\alpha_{d-m}),\ldots,-(\hat\alpha_0+i\hat\alpha_d),\hat\alpha_{d+1},\ldots,\hat\alpha_{d+k} \big]^T \;& \text{if $d$ is odd}\\   
        \big[\hat\alpha_0-i\hat\alpha_d, \ldots, \hat\alpha_m-i\hat\alpha_{d-m},-i\hat\alpha_{\frac{d}{2}}, -(\hat\alpha_m+i\hat\alpha_{d-m}),\ldots,-(\hat\alpha_0+i\hat\alpha_d),\hat\alpha_{d+1},\ldots,\hat\alpha_{d+k} \big]^T\; & \text{if $d$ is even}
 \end{cases}.
\end{equation}
By setting, 
\begin{equation}\label{eq:profpal1}
 Q:=\begin{bmatrix} 
1 & -\lambda &  &  &  &  &  &\\
 & \ddots & \ddots &  &  &  &  &  \\
 &  &  &  -\lambda &  &  &  &  \\
 &  &  & 1 & -\frac{w_1(\lambda)}{\lambda^d} &  &  &  \\
 &  &  &  & 1 & -\frac{w_2(\lambda)}{w_1(\lambda)} &  &  \\
 &  &  &  &  & \ddots & \ddots & \\
 &  &  &  &  &  &  & -\frac{w_k(\lambda)}{w_{k-1}(\lambda)} \\
 &  &  &  &  &  &  &   1 \\
\end{bmatrix}
\;\text{and}\; a:=Q^*\widehat \alpha=\mat{c}a_0\\\vdots \\ a_{d+k}\rix,
\end{equation}
we get $\Lambda Q = e_1^T$ and 
\begin{equation*}
    (Q\otimes I_n)^*H(Q \otimes I_n)  = (e_1a^* \otimes M^* + a e_1^T \otimes M)= \begin{bmatrix}
    a_0M-\overline{a_0}M^* & \overline{a_1}M^* & \cdots & \overline{a_{d+k}M^*}\\
    a_1M &0 &\cdots &0 \\
    \vdots &\vdots & &\vdots \\
    a_{d+k}M &0 & \cdots&0 \\
    \end{bmatrix}.
\end{equation*}
This implies that $a_{1}=\cdots=a_{d+k}=0$, since $H$ is semidefinite and $M$ is invertible, and thus $a=Q^{*} \widehat \alpha=a_{0} e_{1}$. If $d \geq 3$, then  from~\eqref{eq:profpal1}, we have 
\begin{eqnarray*}
a_{1}=0 ~\implies ~  \hat \alpha_{1}-i \hat\alpha_{d-1}=\overline{\lambda}\left(\hat\alpha_{0}-i \hat\alpha_{d}\right) \quad \text{and} \quad 
\overline{a}_{d}=0 ~ \implies ~  \hat\alpha_{0}-i \hat\alpha_{d}=\lambda\left(\hat\alpha_{1}-i \hat\alpha_{d-1}\right),
\end{eqnarray*}
which implies that 
\[
a_{0}=\hat\alpha_{0}-i \hat\alpha_{d}=\lambda\left(\hat\alpha_{1}-i \hat\alpha_{d-1}\right)=\lambda \overline{\lambda}\left(\hat\alpha_{0}-i\hat \alpha_{d}\right)=\lambda \overline{\lambda} a_{0}.
\]
Similarly, when $d=1$, we have 
\[
a_{1}=0 ~\implies~ \hat\alpha_{0}+i \hat\alpha_{1}=-\overline{\lambda}\left(\hat\alpha_{0}-i \hat\alpha_{1}\right) \quad \text { and } \quad \hat\alpha_{0}-i \hat\alpha_{1}=-\lambda\left(\hat\alpha_{0}+i \hat\alpha_{1}\right)
\]
so that $a_{0}=\hat\alpha_{0}-i \hat\alpha_{1}=-\lambda\left(\hat\alpha_{0}+i \hat\alpha_{1}\right)=\lambda \overline{\lambda}\left(\hat\alpha_{0}-i \hat\alpha_{1}\right)=\lambda \overline{\lambda} a_{0}$. 
Finally, when $d=2$, we have 
\begin{eqnarray*}
a_{1}=0 ~\implies~ i \hat\alpha_{1}=-\overline{\lambda}\left(\hat\alpha_{0}-i \hat\alpha_{2}\right) \quad \text{and} \quad 
\overline{a}_{2}=0~\implies~ \hat\alpha_{0}-i \hat\alpha_{2}=-i \lambda \hat\alpha_{1}
\end{eqnarray*}
so that $a_{0}=\hat\alpha_{0}-i \hat\alpha_{2}=-i \lambda \hat\alpha_{1}=\lambda \overline{\lambda}\left(\hat\alpha_{0}-i \hat\alpha_{2}\right)=\lambda \overline{\lambda} a_{0}$.
Thus for any value of $d$, we get $a_0=\lambda \overline{\lambda}a_0$, which implies that $a_0=0$, since $|\lambda|\neq 1$. This implies that $\widehat \alpha = (Q^*)^{-1}a = 0$ and this gives $\hat\alpha_0=\cdots = \hat \alpha_{d+k}=0$, which is a contradiction. 
\eproof

For $T$-palindromic structure, we apply Theorem~\ref{theorem:optimize1} to~\eqref{palT:error1}, to obtain a computable formula for the $T$-palindromic backward error $\eta^{{\rm pal}_T}(G, \lambda)$. More precisely, we have the following result. 

\begin{theorem}\label{thm:sberTpal}
Let $G(z)$ be a T-palindromic RMP of the form~\eqref{rmatrix} and $\lambda \in \C\setminus \{0,\pm 1\}$ satisfying~\eqref{assump}.  Suppose that $M:=(G(\lambda))^{-1} $ exists. Define $m=\lfloor{\frac{d-1}{2}}\rfloor$ and let 
$\psi: \R^{2m+1}\mapsto \R$ be defined by $(t_0,\ldots,t_{2m})\mapsto \lambda_2(F(t_0,\ldots,t_{2m}))$, where $F(t_0,\ldots,t_{2m})$  is as defined in \eqref{eq:defF} for matrices $J,S_0,\ldots,S_m$ in~\eqref{matrix:palJ} and~\eqref{matrix:palHS}.
 Then 
\begin{itemize}
\item The function $\psi$ has a global minimum
\[
\widehat \lambda_2 := \min_{(t_0,\ldots,t_{2m})\in \R^{2m+1}} \psi (t_0,\ldots,t_{2m}) \quad \text{and}\quad 
\eta^{{\rm pal_T}}(G,\lambda) \geq \frac{1}{\sqrt{\widehat \lambda_2}}.
\]
\item Moreover, if the minimum $\widehat \lambda_2$ of $\psi$  is attained 
at $(\hat t_0,\ldots,\hat t_{2m}) \in \R^{2m+1}$ and is a simple eigenvalue of $F(\hat t_0,\ldots,\hat t_{2m}) $, then 
\[
\eta^{{\rm pal_T}}(G,\lambda) = \frac{1}{\sqrt{\widehat \lambda_2}}.
\]
\end{itemize}
\end{theorem}
\proof
In view of~\eqref{palT:error1}, the proof follows from Theorem 
~\ref{theorem:optimize1} if we show that the  matrices $S_0,\ldots, S_m$ satisfy the rank condition, i.e., 
\[
\text{rank}\big((t_0+it_1)S_0+\cdots +(t_{2m-2}+it_{2m-1})S_{m-1}+ t_{2m}S_m\big)\geq 2.
\]
for any nonzero $(t_0,\ldots, t_{2m}) \in \R^{2m+1}$. To see this, let $(t_0,\ldots, t_{2m}) \in \R^{2m+1} \setminus \{0\}$ and consider 
\begin{eqnarray}\label{eq:prof11}
f(t_0,\ldots,t_{2m})&=&(t_0+it_1)S_0  +\cdots +(t_{2m-2}+it_{2m-1})S_{ m-1}+t_{2m}S_m \nonumber \\
&=&\Gamma^{-1}\Big(\Lambda^T\big((t_0+it_1)e_1^T+\cdots + (t_{2m-2}+it_{2m-1}) e_m^T +t_{2m}e_{m+1}^T- t_{2m}e_{d-m+1}^T \nonumber \\
&& - (t_{2m-2}+it_{2m-1}) e_{d-m+2}^T-\cdots -(t_0+it_1)e_{d+1}^T\big)\otimes M^T \nonumber \\
&&+ \big((t_0+it_1)e_1+\cdots + (t_{2m-2}+it_{2m-1}) e_m +t_{2m}e_{m+1}- t_{2m}e_{d-m+1} \nonumber \\
&& - (t_{2m-2}+it_{2m-1}) e_{d-m+2}-\cdots - 
(t_0+it_1)e_{d+1}\big) \Lambda \otimes M\Big)\Gamma^{-1} \nonumber \\
& =&\Gamma^{-1}\left(\Lambda^T\Omega_m \otimes M^T+\Omega_m^T\Lambda \otimes M
\right)\Gamma^{-1},
\end{eqnarray}
where 
\begin{eqnarray*}
\Omega_m &:=&(t_0+it_1)e_1^T+\cdots + (t_{2k-2}+it_{2k-1}) e_k^T +t_{2k}e_{k+1}^T- t_{2k}e_{m-k+1}^T \\
&&- (t_{2k-2}+it_{2k-1}) e_{m-k+2}^T-\cdots -(t_0+it_1)e_{m+1}^T.
\end{eqnarray*}
Thus applying Lemma~\ref{lem:B} in~\eqref{eq:prof11}, we have that $\text{rank}(f(t_0,\ldots,t_{2m}))\geq 2$, since 
$\Omega_m$ and $\Lambda$ are linearly independent, when $\lambda \not\in \{0,\pm 1\}$. 
\eproof

\begin{remark}{\rm \label{rem:optperpal}
A remark similar to Remark~\ref{rem:optpertherm} also holds for palindromic RMPs. When $\bullet = *$ and $\widehat \lambda_{\max}$ is a simple eigenvalue of $J+\sum_{j=0}^{d+k}\widehat t_j H_j$ in Theorem~\ref{thm:mainpalin}, then we first compute the corresponding eigenvector 
$u \in {\C^n}^{d+k+1}$ such that $u^*H_ju=0$ for all $j=0,\ldots,d+k$ and then the optimal perturbation may be obtained from~\cite[Theorem 2.1]{MehMS17}.

When $\bullet=T$ and  $\widehat \lambda_2$ is a simple eigenvalue of $F(\hat t_0,\ldots,t_{2m})$ in Theorem~\ref{thm:sberTpal}, then using~\cite[Lemma 3.4]{MR4404572}, we first compute a vector $u \in {\C^n}^{d+k+1}$ such that $u^TS_ju=0$ for all $j=0,m$ and thus the optimal perturbation may be computed from~\cite[Theorem 2.1]{MehMS17}. 
}
\end{remark}

\section{Structured vs unstructured eigenvalue backward errors}\label{sec:numeric}

In this section, we present some numerical experiments to illustrate the results of this paper by comparing eigenvalue backward errors of RMPs under structure-preserving and arbitrary perturbations. We consider  RMP $G(z)$ with Hermitian and $*$-palindromic structure and compute the eigenvalue backward errors $\eta^{{\rm Herm}}(G,\lambda)$ and $\eta^{{\rm pal}_*}(G,\lambda)$ of some $\lambda \in \mathbb{C} \setminus \{0\}$. In all cases we have used the software package CVX~\cite{GraBBBK08} in MATLAB to solve the associated optimization problem.

\begin{example}\label{ex:herm}{\rm(Hermitian RMP)~
Let $G(z)=A_0+zA_1+\frac{1}{z+1}E_1+\frac{1}{z+2}E_2$ be a $3 \times 3$ Hermitian rational matrix polynomial, where $A_0$, $A_1$, $E_1$ and $E_2$ are randomly generated Hermitian matrices. Eigenvalues of $G(z)$ are given by $-1.7389\pm 1.2830, -0.9013\pm 0.5012, -1.5356\pm 0.0210, 1.2546,-0.2999$, and $-1.8326$.  The Hermitian backward error (Theorem~\ref{thm:Herm}) for the point $\lambda=1.2546+0.5i$ which is close to the eigenvalue $1.2546$ is $0.2888$ while the unstructured backward error (Theorem~\ref{thm:unstrbacerrfor}) $0.0640$ is much smaller as expected.

The spectrum of a Hermitian polynomial $G(z)$ is symmetric to the real axis and this symmetry is preserved to the structure-preserving perturbations. However, this is not the case under arbitrary perturbations to $G(z)$. This is also depicted in~Figure~\ref{figure:herm}. The plot on the left side of Figure~\ref{figure:herm} illustrates the movement of the eigenvalue of the rational matrix $G(z)$ (marked with stars surrounded  by circles) under the homotopic perturbation $G(z)+t\Delta G(z)$ as $t$ moves from $0$ to $1$, 
where $\Delta G(z)=\Delta A_0+z\Delta A_1+\frac{1}{z+1}\Delta E_1+\frac{1}{z+2}E_2$ is the optimal Hermitian perturbation satisfying $\nrm{(\Delta A_0,\Delta A_1,\Delta E_1,\Delta E_2)}=0.2888$. 


On the other hand, the plot on the right side of Figure~\ref{figure:herm} illustrates the same effect to unstructured homotopic perturbations $G(z)+t\Delta G(z)$ as t varies from 0 to 1, where $\Delta G$ is the optimal unstructured perturbation with $\nrm{\Delta G}=0.0640$. It can be observed from the figure that the complex conjugate of $1.2546+0.5i$ is not an eigenvalue of $G(z)+\Delta G(z)$ as it is no more a Hermitian RMP.

\begin{figure}[H]
\centering
\includegraphics[width=\textwidth]{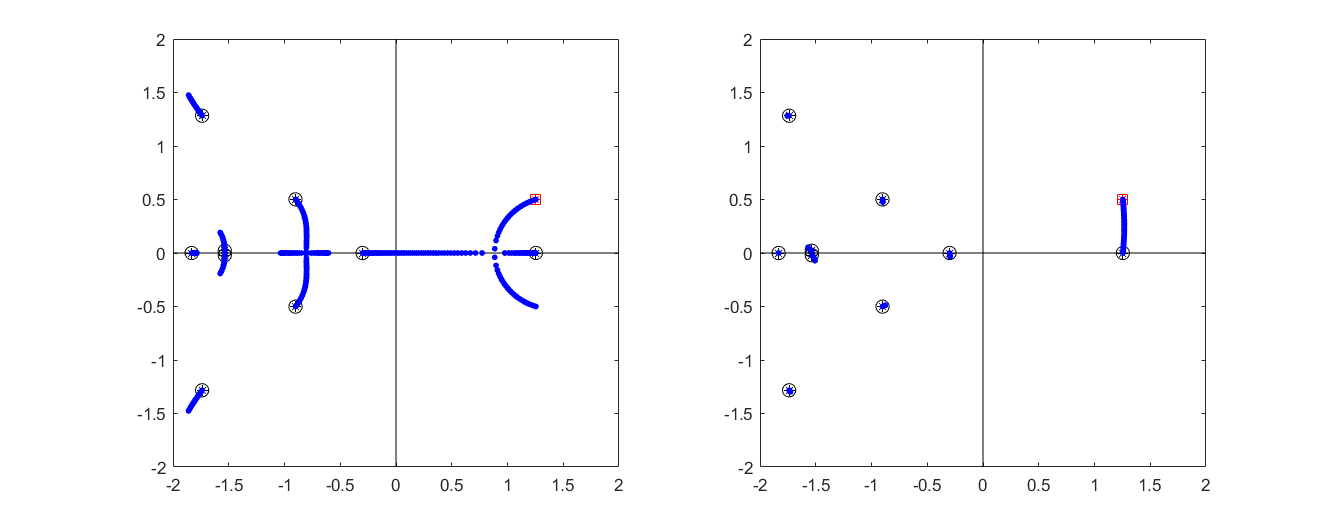}
\caption{Eigenvalue perturbation curves for the Hermitian rational matrix polynomial with respect to structured (left) and unstructured (right) perturbations.}
\label{figure:herm}
\end{figure}

Table~\ref{tab:my_label} compares the structured and unstructured eigenvalue backward errors for different values of $\lambda$. The structured eigenvalue backward error is significantly larger than the unstructured one. As $\lambda$ values approach the eigenvalue $1.2546$ of $G(z)$, as expected, (i) the unstructured backward error tends towards zero, and (ii) the structured backward error also decreases but not as rapidly as the unstructured one.

\begin{table}[]
    \centering
    \begin{tabular}{|c|c|c|}\hline 
        $\lambda$-values & \textbf{$\eta(G,\lambda)$} & \textbf{$\eta^{{\rm Herm}}(G,\lambda)$}\\ 
        &Theorem~\ref{thm:unstrbacerrfor}& Theorem~\ref{thm:Herm} \\ \hline 
        1.2546+0.50i & 0.06405 & 0.28878\\ \hline
        1.2546+0.35i & 0.04525 & 0.28177\\ \hline
        1.2546+0.20i & 0.02601 & 0.27718\\ \hline
        1.2546+0.05i & 0.00652 & 0.27507\\ \hline
        1.2546+0.015i& 0.00196 & 0.27494\\ \hline
        1.2546+0.0015i&0.00019 & 0.27493\\ \hline            
    \end{tabular}
    \caption{Structured and unstructured eigenvalue backward errors of a random Hermitian RMP.}
    \label{tab:my_label}
\end{table}}
\end{example}

\begin{example}{\rm($*$-palindromic RMP)~
Let $P(z)=A+zA^*+\frac{z^2}{z^3+1}E$ be a $3\times 3$ $*$-palindromic RMP, where $A$ and $E$ are randomly generated matrices such that $E^*=E$. 
We consider the point $\lambda=1.0640+0.7033i$ which is close to an eigenvalue $0.8641 + 0.5034i$ of $P(z)$. For this, the computed unstructured (Theorem~\ref{thm:unstrbacerrfor}) and structured (Theorem~\ref{thm:mainpalin}) eigenvalue backward errors are respectively given by $\eta(P,\lambda)=0.02013$ and $\eta^{{\rm pal}_*}(P,\lambda)=0.0625$.

 We experiment similarly to Example~\ref{ex:herm} for $*$-palindromic RMP as well. Note that the eigenvalues of a $*$-palindromic RMP are symmetric to the unit circle. The plot on the left of Figure~\ref{figure:pal} illustrates the movement of the eigenvalues of $P(z)$ under the homotopic perturbations $P(z)+t \Delta P(z)$ as $t$ varies from $0$ to $1$, where $\Delta P(z)$ is an optimal $*$-palindromic perturbation such that $\nrm{(\Delta A_0,\Delta A_1,\Delta E_1)}=0.0625$. The eigenvalues of $P(z)$ are marked with stars surrounded  by circles. To make $\lambda$ (marked with a star surrounded by a square) as an eigenvalue of the perturbed polynomial $P(z)+\Delta P(z)$, the eigenvalue curves originated from the eigenvalues $0.8641+0.5035i$ and $0.9994+0.0343i$ come together on the unit circle and split out to form the pair of eigenvalues  $(\lambda,1/\lambda)$.

On the other hand, the plot on the right of Figure~\ref{figure:pal} illustrates the movement of the eigenvalues under homotopic arbitrary perturbations $P(z)+t\widetilde{\Delta P}(z)$ where $t$ varies from $0$ to $1$, where $\Delta P$ is the optimal unstructured perturbation such that $\nrm{\Delta P}=0.0201$. In this case perturbations do not preserve the palindromic structure of $P(z)$; hence the nearest eigenvalue of $P(z)$ moves to $\lambda$.

\begin{figure}[H]
\centering
\includegraphics[width=\textwidth]{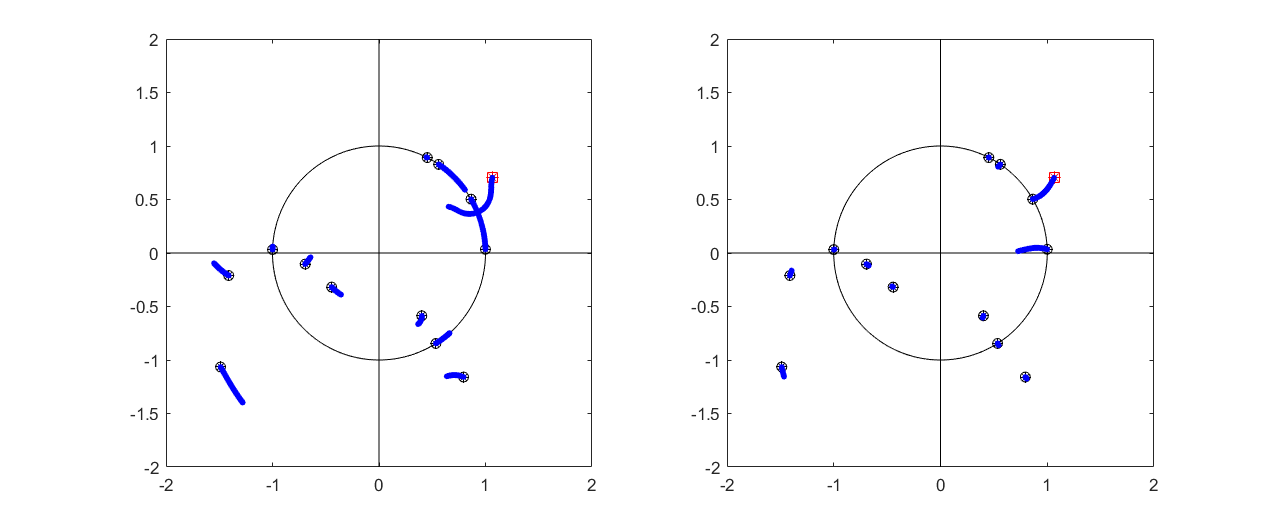}
\caption{Eigenvalue perturbation curves for the $*$-palindromic RMP with respect to structured (left) and unstructured (right) perturbations.}
\label{figure:pal}
\end{figure}
}
\end{example}

We close the section with the following remark.
\begin{remark}{\rm
We note that our framework extends to the general nonlinear eigenvalue problem of the form 
\[
\mathcal T(z)v=0,~v \in \C^n\setminus \{0\},\quad \text{where}\quad 
\mathcal T(z)=f_1(z)\mathcal T_1+\cdots+f_k(z)\mathcal T_k,
\]
the functions $f_1(z),\ldots,f_k(z)$ are meromorphic and assumed to be the given data, and the matrix tuple $(\mathcal T_1,\ldots,\mathcal T_k)$ belong to a structured class $\mathbb S \subseteq (\C^{n,n})^k$.
}
\end{remark}

\section{Conclusion}

We have extended the framework suggested in~\cite{MR3194659} to compute the 
 structured eigenvalue backward errors of RMPs $G(z)$ of the form~\eqref{rmatrix} for the structures mentioned in Table~\ref{tab:my_label}. Numerical experiments suggest a significant difference between structured and unstructured eigenvalue backward errors. 
 As far as we know, no other work has been proposed in the literature to compute structured eigenvalue backward error of RMPs. 
Promising way of future research would be to compute the eigenvalue backward errors under general perturbations in $G(z)$, i.e., when the scalar functions $w_j$s in~\eqref{eq:assumpert} are also subject to perturbation.

\section*{References}
\bibliographystyle{siam}
\bibliography{REP.bib}

\appendix
\section{Proof of Theorem~\ref{thm:unstrbacerrfor}}\label{app:sec1}

\begin{proof}
For a fixed $(\lambda, x) \in \C \times \C^{n}\setminus \{0\}$, if we define the eigenpair backward error of $G(z)$ as
\begin{eqnarray}\label{def:errorep}
            \eta(G,\lambda,x):=\inf\Big\{&\hspace{-.3cm} \nrm{(\Delta A_{0}, \ldots, \Delta A_{d},\Delta E_1, \ldots, \Delta E_k)} \; :\;\Delta A_{0}, \ldots, \Delta A_{d},\Delta E_1, \ldots, \Delta E_k\in \C^{n,n},\nonumber\\
             & \big(\sum_{p=0}^d \lambda^i(A_p-\Delta A_p)+\sum_{j=1}^kw_j(\lambda)(E_j-\Delta E_j)\big)x=0 \Big\},
\end{eqnarray}
then we have
\begin{equation}\label{eq:epairerrorform}
\eta(G, \lambda, x):=\frac{\|(G(\lambda) x \|}{\|x\| \left\|\left(1, \lambda, \ldots,\lambda^d,w_1(\lambda),\ldots w_k(\lambda)\right)\right\|} .
\end{equation}
Indeed, if $\Delta {A_0}, \ldots, \Delta {E_k} \in \mathbb{C}^{n \times n}$ are perturbation matrices such that
\[
\Delta G(\lambda) x:=\sum_{p=0}^{d} \lambda^{p} \Delta {A_p}x+\sum_{j=1}^k w_j(\lambda) \Delta {E_j} x = G(\lambda) x,
\]
that is, $(\lambda, x)$ is an eigenpair of $\sum_{p=0}^{d} z^{p}\left(A_{p}-\Delta A_{p}\right) + \sum_{j=1}^k w_i(z)\left(E_j - \Delta E_j\right)$, then
\begin{eqnarray*}
\|G(\lambda) x\| =\|\Delta G(\lambda) x\| &\leq& \left\|\sum_{p=0}^{d} \lambda^{p} \Delta A_{p} + \sum_{j=1}^k w_j(\lambda)  \Delta E_j\right\| \cdot\|x\|  \nonumber\\
& \leq &\left\|\left(1, \lambda, \ldots,\lambda^d,w_1(\lambda),\ldots w_k(\lambda)\right)\right\| \nrm{\left(\Delta A_{0}, \ldots, \Delta E_{k}\right)}\|x\|,
\end{eqnarray*}
where the last identity follows due to Cauchy-Schwarz inequality. This implies the $``\geq"$ in~\eqref{eq:epairerrorform}. On the other hand, setting
\begin{equation}\label{matrix:unst}
\Delta A_{p}:=\frac{\overline{\lambda}^{p} G(\lambda) x x^{*}}{x^{*} x\left\|\left(1, \lambda, \ldots,\lambda^d,w_1(\lambda),\ldots w_k(\lambda)\right)\right\|^{2}}, \quad \Delta E_{j}:=\frac{\overline{w_j(\lambda)} G(\lambda) x x^{*}}{x^{*} x\left\|\left(1, \lambda, \ldots,\lambda^d,w_1(\lambda),\ldots w_k(\lambda)\right)\right\|^{2}}
\end{equation}
for $p=0,\ldots,d$ and $j=1,\ldots,k$, we have $\nrm{\left(\Delta A_{0}, \ldots, \Delta E_{k}\right)}=\frac{\|(G(\lambda) x \|}{\|x\| \left\|\left(1, \lambda, \ldots,\lambda^d,w_1(\lambda),\ldots w_k(\lambda)\right)\right\|} $ and
\begin{equation}
\Delta G(\lambda) x=\left(\sum_{p=0}^{d} \lambda^{p} \Delta A_{p} + \sum_{j=1}^k w_j(\lambda) \Delta E_j\right)x=G(\lambda) x.
\end{equation}
This proves equality~\eqref{eq:epairerrorform}. Now the proof follows from~~\eqref{eq:epairerrorform} by using the fact that 
\[
\eta(G,\lambda)=\inf_{x\in \C^{n}\setminus \{0\}}\eta(G,\lambda,x)=\frac{\sigma_{\min} (G(\lambda))}{\|(1,\lambda,\ldots,\lambda^d,w_1(\lambda),\ldots,w_k(\lambda))\|}.
\]
\end{proof}

\end{document}